\documentclass[12pt,oneside, reqno]{amsart}
\usepackage[utf8]{inputenc}
\usepackage[english]{babel}
\usepackage{multicol}
\usepackage{setspace}
\usepackage{pifont}

\usepackage{setspace}
\usepackage{hyperref}
\usepackage{amssymb}
\usepackage{amsmath,mathdots}
\usepackage{amsthm}
\usepackage{stmaryrd}
\usepackage{graphicx}
\usepackage{mathrsfs}
\usepackage[noabbrev]{cleveref}
\usepackage{mathtools,leftindex}
\usepackage{tikz-cd}

\usepackage{xcolor}
\hypersetup{
    colorlinks=true,
    linkcolor=blue,
    filecolor=magenta,      
    urlcolor=cyan,
    citecolor=blue,
}
\usepackage{enumerate}

\input xy
\xyoption{all}
\usepackage[color,matrix,arrow]{xy}
\usepackage{tikz}
\usetikzlibrary{bending}
\usetikzlibrary{arrows.meta}
\usetikzlibrary{matrix,arrows,decorations.markings, calc, positioning}
\usetikzlibrary {shapes.geometric}
\usetikzlibrary {intersections}
\usetikzlibrary {decorations.shapes}
\usetikzlibrary {decorations.text}

\usepackage{geometry}
\makeatletter
\providecommand{\leftsquigarrow}{%
  \mathrel{\mathpalette\reflect@squig\relax}%
}
\newcommand{\reflect@squig}[2]{%
  \reflectbox{$\m@th#1\rightsquigarrow$}%
}
\makeatother

\let\epsilon=\relax

\let\F=\relax

\newcommand{\epsilon}{\varepsilon}

\newcommand{\Sub}{\text{Sub}}

\newcommand{\Tmax}{T_{max}}
\newcommand{\Cmax}{K_{max}}

\newcommand{\Wtemp}{\W_{temp}}
\newcommand{\ACtemp}{\mathsf{AC}_{temp}}
\newcommand{\AFtemp}{\mathsf{AF}_{temp}}
\newcommand{\Ctemp}{\mathsf{C}_{temp}}
\newcommand{\Ftemp}{\mathsf{F}_{temp}}
\newcommand{\C}{\mathsf{C}}
\newcommand{\F}{\mathsf{F}}
\newcommand{\W}{\mathsf{W}}
\newcommand{\AF}{\mathsf{AF}}
\newcommand{\AC}{\mathsf{AC}}
\newcommand{\Tr}{\mathrm{Tr}}
\newcommand{\coTr}{\mathrm{CoTr}}
\newcommand{\WFS}{\mathrm{WFS}}
\newcommand{\AFW}{\mathrm{AF}(\W)}
\newcommand{\ACW}{\mathrm{AC}(\W)}

\usepackage{stackengine}
\newlength\rlln
\newlength\rlwd
\newlength\brlwd
\newlength\trlwd
\newlength\lrlwd
\newlength\rrlwd

\def\bstr{\rule{\rlln}{\brlwd}\rule{0pt}{\rlwd}}
\def\tstr{\rule{\rlln}{\trlwd}\rule{0pt}{\rlwd}}
\def\lstr{\ooalign{\rule{\lrlwd}{\rlln}\cr\rule{\rlwd}{0pt}}}
\def\rstr{\ooalign{\rule{\rrlwd}{\rlln}\cr\rule{\rlwd}{0pt}}}
\def\gap{\rule{\dimexpr\rlln-2\rlwd}{0pt}}
\def\pigpen#1#2#3#4{\kern1pt
  \brlwd=#4\rlwd\relax%
  \trlwd=#2\rlwd\relax%
  \lrlwd=#1\rlwd\relax%
  \rrlwd=#3\rlwd\relax%
  \stackengine{\dimexpr\rlln-\rlwd}{%
    \stackengine{0pt}{\bstr}{\lstr\gap\rstr}{O}{c}{F}{F}{L}%
  }{\tstr}{O}{c}{F}{F}{L}\kern1pt%
}

\usepackage{multicol}
\usepackage{pdfpages}
\usepackage{cancel}

\theoremstyle{plain}
\newtheorem{theorem}{Theorem}[section]
\newtheorem{lemma}[theorem]{Lemma}
\newtheorem{proposition}[theorem]{Proposition}
\newtheorem{corollary}[theorem]{Corollary}
\newtheorem*{theorem*}{Theorem}

\theoremstyle{definition}
\newtheorem{definition}[theorem]{Definition}

\newtheorem{example}[theorem]{Example}
\newtheorem{remark}[theorem]{Remark}
\newtheorem{conditions}[theorem]{Condition}

\newtheorem{notation}[theorem]{Notation}

\makeatletter
\let\c@equation\c@theorem
\numberwithin{equation}{section}
\makeatother

\crefname{lemma}{Lemma}{Lemmas}
\crefname{theorem}{Theorem}{Theorems}
\crefname{definition}{Definition}{Definitions}
\crefname{proposition}{Proposition}{Propositions}
\crefname{remark}{Remark}{Remarks}
\crefname{corollary}{Corollary}{Corollaries}
\crefname{example}{Example}{Examples}

\title{Characterizing model structures on finite lattices}
\author{Kristen Mazur, Ang\'elica M. Osorno, Constanze Roitzheim, Rekha Santhanam, Danika Van Niel, and Valentina Zapata Castro}

\address{Department of Mathematics and Statistics, Elon University, Elon, NC 27244, USA}
\email{kmazur@elon.edu}

\address{Department of Mathematics and Statistics, Reed College, Portland, OR 97214, USA}
\email{aosorno@reed.edu}

\address{School of Mathematics, Statistics and Actuarial Science, University of Kent,
Kent CT2 7NF, UK}
\email{c.roitzheim@kent.ac.uk}

\address{Department of Mathematics, IIT Bombay, Powai, Mumbai 400076, India}
\email{reksan@iitb.ac.in}

\address{Department of Mathematics and Statistics, Binghamton University, Binghamton, NY 13902, USA}
\email{dvanniel@binghamton.edu}

\address{Department of Mathematics and Statistics, University of Massachusetts, Amherst, MA 01003, USA}
\email{vzapatacastr@umass.edu}

\subjclass[2020]{18N40; 18B35; 55P99}

\begin{document}

\begin{abstract}
Transfer systems on finite posets have recently been gaining traction as a key ingredient in equivariant homotopy theory. Additionally, they also naturally occur in the data of a model structure. We give a complete characterization of all model category structures on a finite lattice, using transfer systems as our main tool, resulting in new connections between abstract homotopy theory and equivariant methods.
\end{abstract}

\maketitle

\section{Introduction}

Model category structures are a strong foundation for homotopy theory, providing a practical and workable platform to discuss homotopy-related phenomena in different categories. Model structures usually arise in the context of spaces or spectra with extra structure, or elaborate categories arising from homological algebra, which is dependent on  background knowledge of that category. 
Studying model structures on the much simpler category of a finite lattice strips away the specialized technical layer and allows us to focus solely on the model category techniques.

The original definition of a model structure on a category is given via the three classes of morphisms: weak equivalences, fibrations, and cofibrations satisfying a set of strict but naturally occurring axioms.  Due to the nature of the interactions between the classes of morphisms, all data of a model structure can be determined from less than the data given in the classical definition. In particular, a model structure can be completely determined from its weak equivalences and acyclic fibrations, which are the weak equivalences that are also fibrations. 

When the underlying category is a finite lattice, the acyclic fibrations of a model structure form a \textit{transfer system}, a subcategory that contains all objects and is closed under pullbacks. The name ``transfer system'' is a hint to their origins in equivariant homotopy theory. A transfer system defined on the subgroup lattice of an abelian group $G$ encodes a wealth of information related to $G$-equivariant multiplicative structures in homotopy theory. 
Therefore, considering model structures through the lens of transfer systems provides a fresh link between category theory and equivariant homotopy theory. 

Consider, for example, the finite total order $[n]$,  which is  the subgroup lattice of the cyclic group of order $p^n$. The set of all transfer systems on $[n]$ is well understood and can be connected to explicit notions from combinatorics \cite{NinftyOperads}. Moreover, weak equivalence classes of model structures of $[n]$ are given by partitions of $[n]$ \cite{ModelStructuresOnFiniteTotalOrders}, and thus each model structure on $[n]$ is defined by a partition that makes up the weak equivalence classes and a transfer system supported on the partition that makes up the acyclic fibrations. In fact, any partition $\W$ of $[n]$ paired with any transfer system contained in $\W$ gives rise to a model structure on $[n]$ \cite{ModelStructuresOnFiniteTotalOrders}. Therefore, we have the following bijection.
\[\{\text{model structures on } [n]\} \Leftrightarrow \{(\W,T) \mid \text{$\W$ a partition of $[n]$ and $T\subseteq \W$ a transfer system}\} \nonumber \]
We would like a generalization of this result for arbitrary finite lattices.
The appropriate generalization of a partition to any finite lattice is a \emph{wide decomposable subcategory}, \emph{i.e.}, a subcategory that contains all objects and that satisfies the property that if $g \circ f$ is in the subcategory, then so are both $g$ and $f$.  Thus, on $[n]$ every wide decomposable subcategory $\W$ with transfer system $T \subseteq \W$ gives rise to a model structure. For a general lattice, Droz and Zakharevich prove in \cite{DrozZakharevichExtending} that the weak equivalences of a model structure always form a wide decomposable category, but the result observed for $[n]$ does not extend to other lattices. For instance, on the lattice $[1] \times [1] = \Sub(C_{pq})$, every wide decomposable subcategory $\W$ occurs as the weak equivalence class of at least one model structure, but there are transfer systems $T \subseteq \W$ where $T$ does not occur as the acyclic fibrations to a model structure with weak equivalences $\W$, see \cref{ex:square} for more details. Even worse, on the lattice $[2] \times [1]$ there are wide decomposable subcategories that do not occur as the weak equivalence class of a model structure at all, as we show in \cref{ex:middle-arrow}. 

Therefore, in order to fully understand model category structures on finite lattices, we have to answer the following questions.
\begin{itemize}
\item Which wide decomposable subcategories of a finite lattice occur as the weak equivalence class of a model structure?
\item Suppose we know that a wide decomposable subcategory $\W$ is the weak equivalence class of at least one model structure. Which transfer systems $T \subseteq \W$ occur as the acyclic fibrations of a model structure with weak equivalences $\W$?
\end{itemize}

In this paper, we  provide complete and explicit answers to these questions.
The following are our main results. Here we call an arrow $\sigma$ in a finite lattice a \emph{short arrow} if $\sigma=g\circ h$ implies that either $g$ or $h$ must be the identity. In order theory these are called covering relations. 
 
\begin{theorem*}[{\Cref{thm:allaboutw}}]
Let $P$ be a finite lattice and let $Q$ be a wide decomposable  subcategory of $P$. Then $Q$ is the weak equivalence class of at least one model structure if and only if, for all morphisms $f \in Q$, there exists a factorization $f=\sigma_n \circ \sigma_{n-1} \circ \cdots \circ \sigma_1$ into short arrows such that for some $0 \leq k \leq n$, both of the following hold.
\begin{itemize}
\item For any $i \leq k$, all pushouts of $\sigma_i$ are in $\W$.
\item For any $i > k$, all pullbacks of $\sigma_i$ are in $\W$.
\end{itemize}
\end{theorem*}

\begin{theorem*}[{\Cref{thm:AFW-is-lattice}}]
Let $P$ be a finite lattice and $\W \subseteq P$ a wide decomposable subcategory that is the weak equivalence class of at least one model structure. Then the set of transfer systems contained in $\W$ that occur as the acyclic fibrations of a model structure with weak equivalences $\W$ have a maximum element $\AF_{max}$ and a minimum element $\AF_{min}$, such that a given transfer system $T \subseteq \W$ occurs as the acyclic fibrations of a model structure with weak equivalences $\W$ if and only if 
\[
\AF_{min} \subseteq T \subseteq \AF_{max}.
\]
\end{theorem*}

We  furthermore show that $\AF_{max}$ is  the largest transfer system contained in $\W$, and   give an explicit method for constructing $\AF_{min}$.

Finally, we not only  give a complete characterization of model categories on finite posets, but also show that  the characterizing conditions are  practical, feasible, and applicable to a notable class of examples. In addition, our use of transfer systems provides a new, tangible link between the worlds of abstract homotopy theory, category theory, and equivariant topology.

\subsection{Organization}
We begin with the necessary background information on transfer systems and their dual, cotransfer systems, in \Cref{sec:transferbackground}. This includes a study on how transfer and cotransfer systems on a fixed  lattice $P$ form a lattice themselves.  We then provide a roundup of definitions and basic properties of weak factorization systems and model categories in 
\Cref{sec:modelbackground}. 

A reader familiar with transfer systems and model categories can begin reading the article from \Cref{sec:FromTStoMS}, where we use transfer and cotransfer systems  to describe all possible model structures for a fixed subcategory that is the weak equivalence class of at least one model structure. 
\Cref{sec:centers} is then dedicated to characterizing the subcategories that can occur as a weak equivalence class to begin with.
Finally, in \Cref{sec:examples} we apply our results to classify all subcategories that are weak equivalence classes on the lattice $[n]\times [1]$ and describe all model structures on the diamond lattice $[2]^{\ast n}$. We end with an example of a non-modular lattice, namely the pentagon.

\subsection{Acknowledgments}
We would like to thank Scott Balchin, Alvaro Belmonte, and Inna Zakharevich for helpful and encouraging conversations. We would furthermore like to thank Morten Brun for his notation suggestion. We are grateful to the anonymous referee for their valuable comments and helpful suggestions.

AO and CR thank the Isaac Newton Institute for Mathematical Sciences, Cambridge, for support and hospitality during the programme ``Equivariant Homotopy Theory in Context'' where some work on this paper was undertaken. This work was supported by EPSRC grant no EP/Z000580/1. AO, CR, RS, and DVN would like to thank the Isaac Newton Institute for Mathematical Sciences, Cambridge, and Queen's University Belfast for support and hospitality during the programme Operads and calculus, where work on this paper was undertaken. This work was supported by EPSRC grant EP/Z000580/1. KM was partially supported by Elon University's FR\&D summer fellowship. AO was partially supported by NSF grant
DMS–2204365. DVN was partially supported by NSF grant DMS-2135884. RS was partially supported by IIT Bombay CPDA  grant. 

\section{Transfer Systems and Cotransfer Systems on Finite Lattices}\label{sec:transferbackground}

In this section we  introduce posets, lattices, and (co)transfer systems on a finite lattice, which are the main ingredients for describing model structures on finite lattices when considering a lattice as a category. We further summarize results from  \cite{LatticeVia},\cite{algebralattice}, and \cite{RubinDetectingOperads} that show that the set of all (co)transfer systems on a fixed lattice is itself a lattice.  Finally, given a wide subcategory of a lattice, we prove that there exists  a maximum (co)transfer system within that subcategory, a fact which is key to the results in Sections \ref{sec:FromTStoMS} and \ref{sec:centers}.

\begin{definition}\label{def:poset}
    A \emph{poset} $(P,\le)$ consists of a set $P$ together with a binary relation $\le$  that is reflexive, antisymmetric, and transitive. Given two elements $x$ and $y$ in a poset, their \emph{join} $x\vee y$, if it exists, is their least upper bound while their \emph{meet} $x\wedge y$, if it exists, is their greatest lower bound. A \emph{finite lattice} is a finite poset in which every pair of elements has both a join and a meet. This implies that a finite lattice has a maximum and a minimum element.
\end{definition}

Depending on the context, we can represent a lattice as a directed graph with an arrow $a\to b$ whenever the strict relation $a<b$ holds. Alternatively, we can use a Hasse diagram, in which we omit composition arrows.

\begin{example}
    The total order lattice $[n]$ consists of objects $\{0,1,2,3,\dots, n\}$ and the standard order. The lattice $[3]$ is shown below. 
    
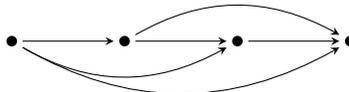
\[
  \begin{tikzpicture}[>=stealth, bend angle=30, baseline=(current bounding box.center)]
    \fill (0,0) circle (2pt);
    \fill (1.5,0) circle (2pt);
    \fill (3,0) circle (2pt);
    \fill (4.5,0) circle (2pt);

    \node (0) at (0,0) {};
    \node (1) at (1.5,0) {};
    \node (2) at (3,0){};
    \node (3) at (4.5,0){};

    \draw[->] (0)--(1);
    \draw[->] (1)--(2);
    \draw[->] (2)--(3);
    \draw[->,bend left] (1) to (3); 
    \draw[->, bend right] (0) to (3);
    \draw[->, bend right, bend angle=60] (0) to (2);
\end{tikzpicture}
    \]
\end{example}

\begin{example}\label{ex:lattice2x} Given two lattices $P$ and $Q$, their product $P\times Q$ is also a lattice, with $(a,b) \le (a',b')$ if and only if $a\le a'$ and $b\le b'$. For example,
    the lattice $[m]\times[n]$ has elements $(i,j)$ for $0\le i \le m$ and $0\le j \le n$. We draw this lattice as a grid with $(0,0)$ as the bottom left vertex and $(m,n)$ as the top right; $[2]\times [1]$ is shown below. 
    
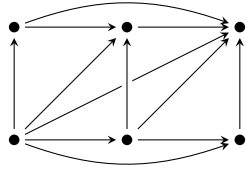
\[
    \begin{tikzpicture}[>=stealth, bend angle=20, baseline=(current bounding box.center)]
    \fill (0,0) circle (2pt);
    \fill (1.5,0) circle (2pt);
    \fill (3,0) circle (2pt);
    \fill (0,1.5) circle (2pt);
    \fill (1.5,1.5) circle (2pt);
    \fill (3,1.5) circle (2pt);

    \node (00) at (0,0){};
    \node (10) at (1.5,0) {};
    \node (20) at (3,0){};
    \node (01) at (0,1.5){};
    \node (11) at (1.5,1.5){};
    \node (21) at (3,1.5){};

    \draw[->] (00)--(01);
    \draw[->] (00)--(10);
    \draw[->] (00)--(11);
    \draw[->] (00)--(21);
    \draw[->, bend right] (00) to (20);
    \draw[white, line width=4pt] (10)--(11);
    \draw[->] (10)--(11);
    \draw[->] (10)--(20);
    \draw[->] (10)--(21);
    \draw[->] (01)--(11);
    \draw[->,bend left] (01) to (21);
    \draw[->] (11)--(21);
    \draw[->] (20)--(21);

    \node (00label) at (-0.5,-0.2){\small $(0,0)$};
    \node (21label) at (3.5,1.7){\small $(2,1)$};
    
    \end{tikzpicture}
    \]
\end{example}

\begin{definition}
    Let $P$ and $Q$ be finite lattices. The \emph{parallel composition} $P*Q$ is the lattice obtained by taking the disjoint union $P\amalg Q$ and identifying the maximum of $P$ with that of $Q$, and similarly for the minima.
\end{definition} 

\begin{example}\label{ex:2iterated}
The lattice $[2]*[2]$ is isomorphic to $[1]\times[1]$. In general, for $n\geq1$, the poset $[2]^{*n}$ has $n$ incomparable elements, together with a maximum and a minimum. 

    For ease of notation, we denote the elements of $[2]^{*n}$ as follows: $\top$ for the maximum, $\bot$ for the minimum, and $1,2,\dots, n$ for the incomparable elements. The only nontrivial relations are $\bot < i$, $i < \top$ for $i=1,\dots,n$, and $\bot < \top$, which is induced by transitivity from the others. Below is $[2]^{*n}$, excluding the relation $\bot < \top$.
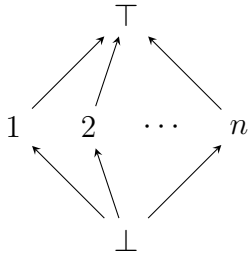
\[
\begin{tikzpicture}[>=stealth, bend angle=20, baseline=(current bounding box.center)]
    \node (bot) at (0,0){$\bot$};
    \node (top) at (0,3){$\top$};
    \node (1) at (-1.5,1.5){$1$};
    \node (2) at (-0.5,1.5){$2$};
    \node at (0.5,1.5){$\cdots$};
    \node (n) at (1.5,1.5){$n$};

    \draw[->] (bot)--(1);
    \draw[->] (bot)--(2);
    \draw[->] (bot)--(n);
    \draw[->] (1)--(top);
    \draw[->] (2)--(top);
    \draw[->] (n)--(top);
\end{tikzpicture}
\]
We revisit this lattice in \Cref{subsec:diamond}.
\end{example}

A poset $(P,\le)$ defines a category with the elements of $P$ as its objects and a unique morphism $a \to b$ whenever $a\leq b$.  Going forward,  we use the terms arrow and morphism interchangeably when referring to the morphisms of categories associated to posets. Given $a$ and $b$ in $P$, there is at most one morphism between $a$ and $b$, \emph{i.e.}, $|\text{Mor}(a,b)|\leq 1$. The identity morphisms are given by reflexivity,  composition is given by transitivity, and antisymmetry implies that the only isomorphisms in this category are the identity morphisms. Conversely, any small category with the property $|\text{Mor}(a,b)| \leq 1$ defines a poset. This has the following consequences.

\begin{enumerate}
    \item Any diagram of composable arrows in a poset category commutes. 
    \item The join of two elements of $P$, if it exists, gives the coproduct in the category, and similar, the meet gives the product.
See the diagram below.
\[\begin{tikzcd}
	a & {a \vee b} \\
	 {a \wedge b} & b
	\arrow[from=2-1, to=1-1]
	\arrow[from=1-1, to=1-2]
	\arrow[from=2-1, to=2-2]
	\arrow[from=2-2, to=1-2]
\end{tikzcd}\]
\item If $P$ is a finite lattice, then the associated category is complete and cocomplete. The limit of any diagram is given by the meet of all the elements in it, and similarly, the colimit is given by their join.
\end{enumerate}

This paper leverages transfer systems, and dually cotransfer systems, to characterize model structures on finite lattices. Transfer systems were originally defined in the context of equivariant operads as a subposet of the subgroup lattice of a finite group satisfying certain properties, see \cite{NinftyOperads,RubinDetectingOperads}. We provide the more general definition of a transfer system on a finite lattice, following the categorical interpretation of \cite{LatticeVia}. We begin with a preliminary definition.

\begin{definition}\label{def:wide} A \emph{wide subcategory} of a category $\mathcal{C}$ is a subcategory that contains all objects. 
\end{definition}

\begin{definition}\label{def:TS}
Let $P$ be a finite lattice, and $\mathcal{R}$ a collection of morphisms. Then
\begin{enumerate}
\item we say $\mathcal{R}$ is a transfer system on $P$ if it is a wide subcategory of $P$ that is closed under pullbacks, and
\item we say $\mathcal{R}$ is a cotransfer system on $P$ if it is a wide subcategory of $P$ that is closed under pushouts.
\end{enumerate}
\end{definition}

\begin{example}\label{ex:tson[2]} A transfer system on $[2]$ is shown below. The presence of the  long arrow implies the short one by the closure under pullbacks.  \[
  \begin{tikzpicture}[>=stealth, bend angle=30, baseline=(current bounding box.center)]
    \fill (0,0) circle (2pt);
    \fill (1.5,0) circle (2pt);
    \fill (3,0) circle (2pt);

    \node (0) at (0,0) {};
    \node (1) at (1.5,0) {};
    \node (2) at (3,0){};

    \draw[->] (0)--(1);
    \draw[->,bend right] (0) to (2); 
    \draw[->,white,bend left] (0) to (2);
\end{tikzpicture}
    \]
In total, there are five transfer systems on $[2]$, see \cite[Example 14]{NinftyOperads}. More generally, for a finite total order $[n]$, there are $Cat(n+1)$ transfer systems, where $Cat(k)$ denotes the $k^{th}$ Catalan number \cite[Theorem 20]{NinftyOperads}.
\end{example}

\begin{example} There are 68 transfer systems and 68 cotransfer systems on the lattice $[2] \times [1]$. Four transfer systems are shown on the left below, and four cotransfer systems are shown on the right.

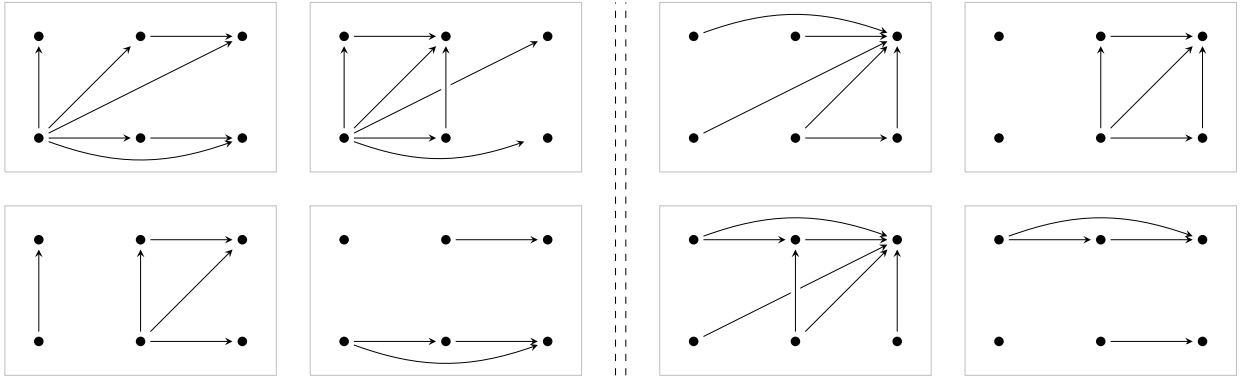
\[
    \resizebox{0.5\textwidth}{!}{
    \begin{tikzpicture}[>=stealth, bend angle=20, baseline=(current bounding box.center)]
    \fill (-3,0) circle (2pt);
    \fill (-1.5,0) circle (2pt);
    \fill (0,0) circle (2pt);
    \fill (-3,1.5) circle (2pt);
    \fill (-1.5,1.5) circle (2pt);
    \fill (0,1.5) circle (2pt);

    \fill (1.5,0) circle (2pt);
    \fill (3,0) circle (2pt);
    \fill (4.5,0) circle (2pt);
    \fill (1.5,1.5) circle (2pt);
    \fill (3,1.5) circle (2pt);
    \fill (4.5,1.5) circle (2pt);

    \fill (-3,-3) circle (2pt);
    \fill (-1.5,-3) circle (2pt);
    \fill (0,-3) circle (2pt);
    \fill (-3,-1.5) circle (2pt);
    \fill (-1.5,-1.5) circle (2pt);
    \fill (0,-1.5) circle (2pt);

    \fill (1.5,-3) circle (2pt);
    \fill (3,-3) circle (2pt);
    \fill (4.5,-3) circle (2pt);
    \fill (1.5,-1.5) circle (2pt);
    \fill (3,-1.5) circle (2pt);
    \fill (4.5,-1.5) circle (2pt);
    
    \node (00tl) at (-3,0) {};
    \node (10tl) at (-1.5,0) {};
    \node (20tl) at (0,0) {};
    \node (01tl) at (-3,1.5) {};
    \node (11tl) at (-1.5,1.5) {};
    \node (21tl) at (0,1.5) {};

    \node (00tr) at (1.5,0) {};
    \node (10tr) at (3,0) {};
    \node (20tr) at (4.3,0) {};
    \node (01tr) at (1.5,1.5) {};
    \node (11tr) at (3,1.5) {};
    \node (21tr) at (4.5,1.5) {};

    \node (00bl) at (-3,-3) {};
    \node (10bl) at (-1.5,-3) {};
    \node (20bl) at (0,-3) {};
    \node (01bl) at (-3,-1.5) {};
    \node (11bl) at (-1.5,-1.5) {};
    \node (21bl) at (0,-1.5) {};

    \node (00br) at (1.5,-3) {};
    \node (10br) at (3,-3) {};
    \node (20br) at (4.5,-3) {};
    \node (01br) at (1.5,-1.5) {};
    \node (11br) at (3,-1.5) {};
    \node (21br) at (4.5,-1.5) {};

    \draw[->] (00tl) -- (01tl);
    \draw[->] (00tl)--(10tl);
    \draw[->] (00tl)--(11tl);
    \draw[->] (10tl)--(20tl);
    \draw[->] (11tl)--(21tl);
    \draw[->] (00tl)--(21tl);
    \draw[->, bend right] (00tl) to (20tl);
    \draw[->, bend left,white] (01tl) to (21tl);

    \draw[->] (00tr)--(10tr);
    \draw[->] (00tr)--(01tr);
    \draw[->] (00tr)--(11tr);
    \draw[->] (00tr)--(21tr);
    \draw[->, bend right] (00tr) to (20tr);
    \draw[->] (01tr)--(11tr); 
    \draw[white, line width=4pt] (10tr)--(11tr);
    \draw[->] (10tr)--(11tr);

    \draw[->] (00bl)--(01bl);
    \draw[->] (10bl)--(11bl);
    \draw[->] (10bl)--(20bl);
    \draw[->] (10bl)--(21bl);
    \draw[->] (11bl)--(21bl);

    \draw[->] (00br)--(10br);
    \draw[->, bend right] (00br) to (20br);
    \draw[->] (10br)--(20br);
    \draw[->] (11br)--(21br);

    \draw[lightgray] (-3.5,-0.5)--(0.5,-0.5)--(0.5,2)--(-3.5,2)--(-3.5,-0.5);
    \draw[lightgray] (1,-0.5)--(5,-0.5)--(5,2)--(1,2)--(1,-0.5);
    \draw[lightgray] (-3.5,-3.5)--(0.5,-3.5)--(0.5,-1)--(-3.5,-1)--(-3.5,-3.5);
    \draw[lightgray] (1,-3.5)--(5,-3.5)--(5,-1)--(1,-1)--(1,-3.5);

    \draw[dashed] (5.5,-3.5)--(5.5,2);
\end{tikzpicture}}
    \resizebox{0.5\textwidth}{!}{
    \begin{tikzpicture}[>=stealth, bend angle=20, baseline=(current bounding box.center)]
    \fill (-3,0) circle (2pt);
    \fill (-1.5,0) circle (2pt);
    \fill (0,0) circle (2pt);
    \fill (-3,1.5) circle (2pt);
    \fill (-1.5,1.5) circle (2pt);
    \fill (0,1.5) circle (2pt);

    \fill (1.5,0) circle (2pt);
    \fill (3,0) circle (2pt);
    \fill (4.5,0) circle (2pt);
    \fill (1.5,1.5) circle (2pt);
    \fill (3,1.5) circle (2pt);
    \fill (4.5,1.5) circle (2pt);

    \fill (-3,-3) circle (2pt);
    \fill (-1.5,-3) circle (2pt);
    \fill (0,-3) circle (2pt);
    \fill (-3,-1.5) circle (2pt);
    \fill (-1.5,-1.5) circle (2pt);
    \fill (0,-1.5) circle (2pt);

    \fill (1.5,-3) circle (2pt);
    \fill (3,-3) circle (2pt);
    \fill (4.5,-3) circle (2pt);
    \fill (1.5,-1.5) circle (2pt);
    \fill (3,-1.5) circle (2pt);
    \fill (4.5,-1.5) circle (2pt);
    
    \node (00tl) at (-3,0) {};
    \node (10tl) at (-1.5,0) {};
    \node (20tl) at (0,0) {};
    \node (01tl) at (-3,1.5) {};
    \node (11tl) at (-1.5,1.5) {};
    \node (21tl) at (0,1.5) {};

    \node (00tr) at (1.5,0) {};
    \node (10tr) at (3,0) {};
    \node (20tr) at (4.5,0) {};
    \node (01tr) at (1.5,1.5) {};
    \node (11tr) at (3,1.5) {};
    \node (21tr) at (4.5,1.5) {};

    \node (00bl) at (-3,-3) {};
    \node (10bl) at (-1.5,-3) {};
    \node (20bl) at (0,-3) {};
    \node (01bl) at (-3,-1.5) {};
    \node (11bl) at (-1.5,-1.5) {};
    \node (21bl) at (0,-1.5) {};

    \node (00br) at (1.5,-3) {};
    \node (10br) at (3,-3) {};
    \node (20br) at (4.5,-3) {};
    \node (01br) at (1.5,-1.5) {};
    \node (11br) at (3,-1.5) {};
    \node (21br) at (4.5,-1.5) {};
    \draw[->] (00tl)--(21tl);
    \draw[->, bend left] (01tl) to (21tl);
    \draw[->] (11tl)--(21tl);
    \draw[->] (10tl)--(20tl);
    \draw[->] (10tl)--(21tl);
    \draw[->] (20tl) to (21tl);

    \draw[->] (10tr)--(11tr);
    \draw[->] (10tr)--(20tr);
    \draw[->] (10tr)--(21tr);
    \draw[->] (11tr)--(21tr);
    \draw[->] (20tr)--(21tr);

    \draw[->] (00bl)--(21bl);
    \draw[->] (10bl)--(21bl);
    \draw[->] (11bl)--(21bl);
    \draw[white, line width=4pt] (10bl)--(11bl);
    \draw[->] (10bl)--(11bl);
    \draw[->] (01bl)--(11bl);
    \draw[->, bend left] (01bl) to (21bl);
    \draw[->] (20bl)--(21bl);

    \draw[->] (01br)--(11br);
    \draw[->, bend left] (01br) to (21br);
    \draw[->] (10br)--(20br);
    \draw[->] (11br)--(21br);
    \draw[->, bend right,white] (00br) to (20br);

    \draw[lightgray] (-3.5,-0.5)--(0.5,-0.5)--(0.5,2)--(-3.5,2)--(-3.5,-0.5);
    \draw[lightgray] (1,-0.5)--(5,-0.5)--(5,2)--(1,2)--(1,-0.5);
    \draw[lightgray] (-3.5,-3.5)--(0.5,-3.5)--(0.5,-1)--(-3.5,-1)--(-3.5,-3.5);
    \draw[lightgray] (1,-3.5)--(5,-3.5)--(5,-1)--(1,-1)--(1,-3.5);

    \draw[dashed] (-4,-3.5)--(-4,2);
\end{tikzpicture}}
\]

\end{example}

The set of all (co)transfer systems on a finite lattice is a poset via the usual inclusion, \emph{i.e.}, $T_1 \le T_2$ if and only if every arrow in $T_1$ is also in $T_2$.  

\begin{definition}
Let $P$ be a finite lattice. Define $\Tr(P)$ to be the poset of transfer systems on $P$, and define $\coTr(P)$ to be the poset of cotransfer systems on $P$.
\end{definition}

Further, $\Tr(P)$ and $\coTr(P)$ are lattices with meets and joins defined as follows. By definition, if $T_1$ and $T_2$ are transfer systems on $P$, their intersection $T_1 \cap T_2$ is again a transfer system, and thus gives the meet. The analogous statement is true for cotransfer systems. Constructing the join in $\Tr(P)$ and  $\coTr(P)$  requires more work. We begin by introducing some notation.

\begin{definition}\label{def:join}
Let $P$ be a finite lattice and $S$ a subset of the relations in $P$. Then $\left< S \right>$ is the smallest transfer system containing $S$, and $(\!( S )\!)$ is the smallest cotransfer system containing $S$.
\end{definition}

We can define $\left<S\right>$  as the intersection of all transfer systems on $P$ containing $S$, which justifies its existence. Dually, $(\!( S )\!)$ is the intersection of all cotransfer systems containing $S$.

 \begin{remark}\label{rem:generation}
 Let $P$ be a finite lattice and $S \subseteq P$ a subset with all isomorphisms. By \cite[Theorem A.2]{RubinDetectingOperads}, we can obtain $\left<S\right>$ by closing $S$ under the following operations.
\begin{enumerate}
\item Take the closure of $S$ under pullbacks. 
\item Take the closure of the collection obtained in the previous step under composition.
\end{enumerate}
In other words, 
\[ \left<S\right> = \{s_n \circ s_{n-1} \circ \cdots \circ s_1 \, | \, \mbox{$n\geq 0$, $s_i$ is a pullback of an element in $S$ for all $i=1,\dots n$}\}.\]

This construction is similar to \cite[Lemma 3.6]{MORSVZ}. 
\end{remark}

We omit the dual construction for cotransfer systems since we do not use it  in this article.

\begin{example} Consider  the following subset $S$ of  $[2]\times[1]$.
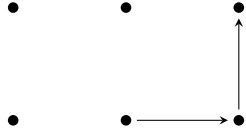
\[
    \begin{tikzpicture}[>=stealth, bend angle=20, baseline=(current bounding box.center)]
    \fill (0,0) circle (2pt);
    \fill (1.5,0) circle (2pt);
    \fill (3,0) circle (2pt);
    \fill (0,1.5) circle (2pt);
    \fill (1.5,1.5) circle (2pt);
    \fill (3,1.5) circle (2pt);

    \node (00) at (0,0){};
    \node (10) at (1.5,0) {};
    \node (20) at (3,0){};
    \node (01) at (0,1.5){};
    \node (11) at (1.5,1.5){};
    \node (21) at (3,1.5){};

    \draw[->] (10)--(20);
    \draw[->] (20)--(21);
    \draw[->, white, bend left] (01) to (21);
    \draw[->, white, bend right] (00) to (20);
    
    \end{tikzpicture}
    \]
    The transfer system $\left<S\right>$ is shown below.
    \[
    \begin{tikzpicture}[>=stealth, bend angle=20, baseline=(current bounding box.center)]
    \fill (0,0) circle (2pt);
    \fill (1.5,0) circle (2pt);
    \fill (3,0) circle (2pt);
    \fill (0,1.5) circle (2pt);
    \fill (1.5,1.5) circle (2pt);
    \fill (3,1.5) circle (2pt);

    \node (00) at (0,0){};
    \node (10) at (1.5,0) {};
    \node (20) at (3,0){};
    \node (01) at (0,1.5){};
    \node (11) at (1.5,1.5){};
    \node (21) at (3,1.5){};

    \draw[->] (00)--(01);
    \draw[->] (10)--(11);
    \draw[->] (10)--(20);
    \draw[->] (10)--(21);
    \draw[->] (20)--(21);
    \draw[->, white, bend left] (01) to (21);
    \draw[->, white, bend right] (00) to (20);
    
    \end{tikzpicture}
    \]
\end{example}

\begin{theorem}[{\cite[Proposition 3.1]{RubinOperadicLifts}}] \label{lem: joins} Let $P$ be a finite lattice, then the posets $\Tr(P)$ and $\coTr(P)$ are lattices. More precisely, for transfer systems $T_1$ and $T_2$, the join $T_1\vee T_2$ is given by 
\[T_1\vee T_2 =\left<T_1 \cup T_2\right>. \]
Dually, if $K_1$ and $K_2$  are cotransfer systems on $P$, then their join is given by
\[K_1 \vee K_2=(\!( K_1 \cup K_2 )\!).\]
\end{theorem}

\begin{remark}
Because $T_1$ and $T_2$ are transfer systems and therefore closed under pullbacks, 
by \Cref{rem:generation}, we only need to close $T_1 \cup T_2$ under composition to obtain $T_1 \vee T_2$.  Similarly for cotransfer systems.
\end{remark}

Now that we have defined the meet and join operations, we have proved the following result.

\begin{example}
The lattice of the five transfer systems on $[2]$ is a pentagon.
More generally, $\Tr([n])$ is the Tamari lattice, \emph{i.e.}, the $0$ and $1$-cells of the associahedron for multiplying $n+2$ elements \cite[Theorem 25]{NinftyOperads}.
\end{example}

We end this section by discussing the collection of (co)transfer systems of a finite lattice $P$ that lie within a given wide subcategory $Q$ of $P$.

\begin{lemma}\label{lem:vee-in-W}
    Let $Q$ be a wide  subcategory of a finite lattice $P$. If $T_1,T_2\subseteq Q$ are transfer systems, then $T_1 \vee T_2 \subseteq Q$. Dually, if $K_1,K_2 \subseteq Q$ are cotransfer systems, then $K_1 \vee K_2 \subseteq Q$.
\end{lemma}

\begin{proof}
    Since $T_1\vee T_2$ is constructed by closing $T_1\cup T_2$ under composition and $Q$ is a subcategory,  $T_1\vee T_2$ must be contained in $Q$. The dual proof is the same.
\end{proof}

\begin{definition}[$\Tmax$ and $\Cmax$]\label{def:Tmax} Let $Q$ be a wide  subcategory of a finite lattice $P$. Define $\Tmax$ to be the join of all transfer systems on $P$ that are contained in $Q$. Dually, define $\Cmax$ to be the join of all cotransfer systems on $P$ that are contained in $Q$.
\end{definition}

While $\Tmax$ and $\Cmax$ depend on $Q$, we do not include $Q$ in the notation since in subsequent sections we  work with a fixed $Q$. By construction, $\Tmax$ is a maximum along the collection of transfer systems on $P$ contained within $Q$, and similarly for $\Cmax$. This property is leveraged in Sections \ref{sec:FromTStoMS} and \ref{sec:centers}.

\begin{remark}\label{rem:maximality}
Given a morphism $f$ that is in $Q$ but not in $\Tmax$, the transfer system on $P$ generated by $T_{max}$ and $f$, $\left<\Tmax \cup \{f\}\right>$,  by maximality is not completely contained in $Q$. Therefore any morphism in $Q$ that is closed under pullbacks must be in $\Tmax$.
\end{remark}

 The transfer system $\Tmax$ in $Q$ satisfies further useful properties when $Q$ is also decomposable. 

\begin{definition}
    A \emph{decomposable subcategory} of a category $\mathcal{C}$ is a subcategory $\mathcal{D}$ such that for all composable morphisms $f,g\in \mathcal{C}$ such that $g\circ f$ is in $\mathcal{D}$, then both $f$ and $g$ are in $\mathcal{D}$.
\end{definition}

In the case of $Q$ being decomposable, $\Tmax$ is a \emph{saturated} transfer system and $\Cmax$ is a saturated cotransfer system. This property will be helpful when computing $\Tmax$ and $\Cmax$ for a given $Q$. The terminology ``saturated'' was introduced in \cite{RubinDetectingOperads}, where it was established that the transfer systems associated to linear isometries operads are saturated.

\begin{definition}\label{defn:saturated}
A transfer system $T$ on a finite lattice $P$ is \emph{saturated} if whenever $x \leq y \leq z$ and $x \rightarrow z$ is in $T$, then so is $y \rightarrow z$. A cotransfer system $K$ on a finite lattice $P$ is \emph{saturated} if whenever $x \leq y \leq z$ and $x \rightarrow z$ is in $K$, then so is $x \rightarrow y$.
\end{definition}

\begin{remark}\label{rem:saturated=decomposable}
A (co)transfer system is saturated if and only if it is decomposable. Indeed, if $x\leq y \leq z$ and $x\to z$ is in a transfer system $T$, closure under pullbacks implies $x\to y$ is in $T$. Similarly, for a cotransfer system $K$, since $x \rightarrow z$ implies $y \rightarrow z$ in $K$.
\end{remark}

The above gives us the following result.

\begin{proposition}\label{prop:Tmax-saturated} Given a wide decomposable subcategory $Q$ of a finite poset $P$,
the transfer system $\Tmax$ and the cotransfer system $\Cmax$ are saturated.
\end{proposition}

\begin{proof}
Assume $x\leq y \leq z$ and $x \rightarrow z$ is in $\Tmax$. Let $f\colon y \rightarrow z$. Since $Q$ is decomposable, it follows that $f$ is in $Q$.  We will show $f$ is in $\Tmax$ by showing that  $\left< \Tmax \cup \{f\} \right>$  is  contained in $Q$, and using \Cref{rem:maximality}.

Since $\left< \Tmax \cup \{f\} \right>$ is obtained by closing the set $\Tmax \cup \{f\}$ first under pullbacks and then under compositions, and $Q$ is closed under compositions, it suffices then to show that $\left< \Tmax \cup \{f\} \right>$ is closed under pullbacks. Any pullback of $f$ along a map $a \rightarrow z$ fits into the following diagram.
\[
\xymatrix{ a \ar[r] & z \\
 a \wedge y \ar[r] \ar[u] & y \ar[u]^f \\
  a \wedge x\ar[r] \ar[u] \ar@/^2pc/[uu] &  x \ar[u] \ar@/_2pc/[uu] \\
}
\]

As the arrow $x \rightarrow z$ is in $\Tmax$, so is its pullback $a \wedge x \rightarrow a$. In particular, $a \wedge x \rightarrow a$ is in $Q$ and since $Q$ is decomposable, then   $a \wedge y \rightarrow a$ must also be in $Q$. Thus, we see that pullbacks of elements of $\Tmax \cup \{f\}$ are in $Q$.

We conclude that $\left< \Tmax \cup \{f\} \right>$ is a transfer system in $Q$, which by maximality implies that $\left< \Tmax \cup \{f\} \right> \subseteq \Tmax$, thus they are equal and $\Tmax$ is saturated.

The proof for $\Cmax$ is similar. 
\end{proof}

Saturation of $\Tmax$ implies that it can be constructed from the short arrows in $Q$, which in order theory, are called covering relations.

\begin{definition}\label{def:short-arrows}
A morphism $\sigma$ in a lattice $P$ is called a \emph{short arrow} if $\sigma=\sigma_2 \circ \sigma_1$ implies that either $\sigma_1$ or $\sigma_2$ is the identity.
\end{definition}

\begin{proposition}\label{prop:Tmax-closure} Let $Q$ be a wide  subcategory of a finite lattice $P$, and let $S$ be the set
\[
S=\{ f \in Q \,\,|\,\, f \,\,\mbox{is a short arrow whose pullbacks are in $Q$ or an isomorphism}\}.
\]
Then  the transfer system $\Tmax$ is the closure of the set $S$ under composition.
\end{proposition}

\begin{proof}
    Let $f\in S$, meaning a short arrow in $Q$ such that all its pullbacks are also in $Q$. Then the transfer system $\left< \{ f\} \right>$ is the closure under composition of a set that lies within $Q$, and since $Q$ is a subcategory, we have that $\left< \{ f\} \right>$ is contained in $Q$. It follows that $\left< \{ f\} \right> \subseteq \Tmax$, and in particular, $f\in \Tmax$. Thus $\Tmax$ contains the set $S$, and therefore it contains the closure under composition of S.

    Conversely, let $f\in \Tmax$. By \cref{prop:Tmax-saturated} and \cref{rem:saturated=decomposable} $f=\sigma _n \circ \dots \circ \sigma_1$ where each $\sigma_i$ is a short arrow and $\sigma_i \in T_{max}$ for all $i$. Since $\Tmax$ is closed under pullbacks and is contained within $Q$, all pullbacks of each $\sigma_i$ are in $Q$. It follows that $\sigma_i\in S$, and hence $f$ is contained in the closure under composition of $S$. 
\end{proof}

A dual argument proves the following.

\begin{proposition}\label{prop:Kmax-closure} Let $Q$ be a wide  subcategory of a finite lattice $P$, and let $S$ be the set
\[
S=\{ f \in Q \,\,|\,\, f \,\,\mbox{is a short arrow whose pushouts are in $Q$ or an isomorphism}\}.
\]
Then the cotransfer system $\Cmax$ is the closure of the set $S$ under composition.
\end{proposition}

\section{Weak Factorization Systems and Model Structures}\label{sec:modelbackground}

Weak factorization systems arise naturally in  many areas of category theory, homotopy theory, and algebra. In the context of transfer systems and this article, weak factorization systems are the  preferred way to present model category structures. Weak factorization systems are understood in terms of lifting properties, so we begin with the following definitions.

\begin{definition}
    For any two morphisms $i\colon a \to b$ and $p \colon x \to y$ in a category $\mathcal{C}$, we say that $i$\emph{ has the left lifting property (LLP) with respect to} $p$, or $p$\emph{ has the right lifting property (RLP) with respect to} $i$, denoted $i \boxslash p$, if for all commutative squares as shown below there exists a lift $h\colon b \to x$ which makes the resulting diagram commute. 
\centerline{
\begin{tikzpicture}
    \node (a) at (0,2) {$a$};
    \node (x) at (2,2) {$x$};
    \node (b) at (0,0) {$b$};
    \node (y) at (2,0) {$y$};

    \draw[->] (a) -- (x);
    \draw[->] (a) -- (b) node[midway, left] {$i$};
    \draw[->] (x) -- (y) node[midway, right] {$p$};
    \draw[->] (b) -- (y);
    \draw[dashed,->] (b) -- (x) node[midway, right] {$\exists h$};
\end{tikzpicture}
}
For any class $\mathcal{S}$ of morphisms in $\mathcal{C}$ we write
\begin{align*}
    \mathcal{S}^\boxslash = \{ g \in \text{Mor}(\mathcal{C}) \, & | \, f \boxslash g \text{ for all } f \in \mathcal{S} \}, \text{ and} \\
    {}^\boxslash\mathcal{S} = \{ f \in \text{Mor}(\mathcal{C}) \, & | \, f \boxslash g \text{ for all } g \in \mathcal{S} \}.
\end{align*}
Given two classes of morphisms $\mathcal{S}$ and $\mathcal{T}$, $\mathcal{S} \subseteq {}^\boxslash \mathcal{T}$ if and only if $\mathcal{T} \subseteq \mathcal{S}^\boxslash$. We write $\mathcal{S} \boxslash \mathcal{T}$ when this holds.
\end{definition}

\begin{definition}
    A \emph{weak factorization system} on a category $\mathcal{C}$ consists of a pair of classes of morphisms $(L,R)$  such that:
    \begin{enumerate}
        \item every morphism $f\in\mathcal{C}$ can be factored as $f=p\circ i$ where $i\in L$ and $p\in R$, and
        \item $L={}^\boxslash R$ and $R=L^\boxslash$.
    \end{enumerate}
    Given a weak factorization system $(L,R)$, we call $L$ the \textit{left set} and $R$ the \textit{right set}.
\end{definition}

\begin{definition}
We denote the collection of weak factorization systems on $\mathcal{C}$ by $\WFS(\mathcal{C})$.
    This collection  has a partial order given by inclusion of the right set, \emph{i.e.}, $(L,R) \leq (L',R')$ if and only if $R \subseteq R'$. Since $L'={}^\boxslash R'$ and $L={}^\boxslash R$, it follows that $R \subseteq R'$ if and only if $L'\subseteq L$.  
\end{definition}

  Given a category $\mathcal{C}$ and a weak factorization system $(L,R)$, the left set $L$ is  closed under pushouts and contains all isomorphisms, as $L={}^\boxslash R$. Dually, $R$ is  closed under pullbacks and contains all isomorphisms, as $R=L^\boxslash$. Thus, in the case of a finite lattice $P$, if $(L,R)$ is a weak factorization system, by definition $R$ is a transfer system and $L$ is a cotransfer system. The following result tells us that in the case of a finite lattice, every transfer system and every cotransfer system is part of a weak factorization system.

\begin{theorem}[{\cite[Theorem 4.13]{LatticeVia}}]\label{rem:halftransfer}
Let $P$ be a finite lattice. Then the assignment $(L,R) \mapsto R$ gives an isomorphism of posets 
\[\WFS(P) \longrightarrow \Tr(P),\]
with inverse $T\mapsto ({}^\boxslash T, T)$. 

A dual statement holds for cotransfer systems, with the assignment $(L,R) \to L$ giving an order-reversing bijection
\[\WFS(P) \longrightarrow \coTr(P),\]
with inverse $K \mapsto (K,K^\boxslash)$.

We thus have the following bijections, where the bijections between $\WFS(P)$ and $\Tr(P)$ are order-preserving and all others are order-reversing.
\[
\xymatrix{ ({}^\boxslash T,T)  &  & \WFS(P) \ar[dl] \ar[dr] & & (K,K^\boxslash) \\
T \ar@{|->}[u]  & \Tr(P) \ar@<1ex>[rr]^{{}^\boxslash (-)} \ar[ur] & & \coTr(P) \ar@<1ex>[ll]^{(-)^\boxslash} \ar[ul] & K \ar@{|->}[u] \\
}
\]
    
\end{theorem}

In order to construct a weak factorization system on a finite lattice $P$, we can start with a transfer system $T$ and construct  ${}^\boxslash T$, or dually, we can start with a cotransfer system $K$ and construct $K^\boxslash$. Explicit constructions of these were developed in \cite{LatticeVia} using the downward extension, which we now define.

\begin{definition}[Downward extension]\cite[Definition 4.14]{LatticeVia} Let $T$ be a transfer system on a finite lattice $P$. We define the \emph{downward extension} of $T$ to be
    \[
    \mathcal{E}_d(T) = \{ z \to y \, | \, \text{there exists} \, \, x \in P \, \, \text{such that} \, ,\ z \leq x < y \, \, \text{and} \, \, x \to y \in T \}.
    \]
\end{definition}
\begin{example}\label{ex:downward-extension} On $[m]\times [n]$, assuming the same orientation as in \Cref{ex:lattice2x} with $(0,0)$ at the bottom left and $(m,n)$ at the top right, the downward extension of a transfer system $T$ consists of all non-identity morphisms in $T$ along with all morphisms that are obtained by dragging the source of each morphism of $T$ down and to the left, while keeping the target stationary. For example, the following diagram shows a transfer system on $[2]\times[1]$ and its respective downward extension.

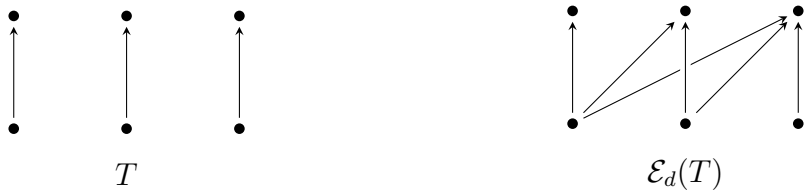
\[
    \begin{minipage}{0.45\textwidth}
    \centering
        \begin{tikzpicture}[>=stealth, bend angle=20, baseline=(current bounding box.center)]
    \fill (0,0) circle (2pt);
    \fill (1.5,0) circle (2pt);
    \fill (3,0) circle (2pt);
    \fill (0,1.5) circle (2pt);
    \fill (1.5,1.5) circle (2pt);
    \fill (3,1.5) circle (2pt);
    
    \node (00) at (0,0) {};
    \node (10) at (1.5,0) {};
    \node (20) at (3,0) {};
    \node (01) at (0,1.5) {};
    \node (11) at (1.5,1.5) {};
    \node (21) at (3,1.5) {};

    \draw[->] (00)--(01);
    \draw[->] (10)--(11);
    \draw[->] (20)--(21);
    \draw[->, white, bend right] (00) to (20); 
\end{tikzpicture}

$T$
    \end{minipage}
    \begin{minipage}{0.45\textwidth}
    \centering
        \begin{tikzpicture}[>=stealth, bend angle=20, baseline=(current bounding box.center)]
    \fill (0,0) circle (2pt);
    \fill (1.5,0) circle (2pt);
    \fill (3,0) circle (2pt);
    \fill (0,1.5) circle (2pt);
    \fill (1.5,1.5) circle (2pt);
    \fill (3,1.5) circle (2pt);
    
    \node (00) at (0,0) {};
    \node (10) at (1.5,0) {};
    \node (20) at (3,0) {};
    \node (01) at (0,1.5) {};
    \node (11) at (1.5,1.5) {};
    \node (21) at (3,1.5) {};

    \draw[->] (00)--(01);
    \draw[->] (00)--(21);
    \draw[->] (00)--(11);
    \draw[white, line width=4pt] (10)--(11);
    \draw[->] (10)--(11);
    \draw[->] (10)--(21);
    \draw[->] (20)--(21);
    \draw[->,white, bend right] (00) to (20); 
\end{tikzpicture}

$\mathcal{E}_d(T)$
    \end{minipage}
\]
\end{example}
Analogously, we can define an extension for cotransfer systems in the following way.
\begin{definition}[Upward extension] Let $K$ be a cotransfer system on a finite lattice $P$. We define the \emph{upward extension} of $K$ to be
    \[
    \mathcal{E}_u(K) = \{ z \to y \, | \, \text{there exists} \, \, x \in P \, \, \text{such that} \, ,\ z < x \leq y \, \, \text{and} \, \, z \to x \in K \}.
    \]
\end{definition}

The following result gives an explicit way of computing ${}^\boxslash T$ and $K^\boxslash$ using downward and upward extensions.

\begin{proposition}[{\cite[Proposition 4.15]{LatticeVia}}]\label{prop:liftformula}
Let $T$ be a transfer system and $K$ a cotransfer system on a finite lattice $P$, then 
\[
{}^\boxslash T=  \mathcal{E}_d(T)^c \,\,\,\mbox{and}\,\,\,K^\boxslash=\mathcal{E}_u(K)^c.
\]
\end{proposition}

We now define model categories. The definition presented here is a reformulation of the original definition by Quillen given by Joyal and Tierney \cite{HoAlgQuillen, JT}.

\begin{definition}\label{def:WFSmodel}
Let $\mathcal{C}$ be a category with all finite limits and colimits.
A \emph{premodel category} on $\mathcal{C}$ consists of two weak factorization systems 
\[
(\AC,\F) \text{ and } (\C, \AF) \,\,\,\mbox{such that}\,\,\, \AC \subseteq \C, \,\,\mbox{or equivalently}\,\,\, \AF \subseteq \F.
\]
If in addition the class of morphisms $\W := \AF \circ \AC$ satisfies the two-out-of-three axiom, then we call the data given by $\W, \AC, \C, \AF$ and $\F$ a \emph{model structure}. We call the morphisms in $\W$ the \emph{weak equivalences}, $\F$  the \emph{fibrations}, $\C$ the \emph{cofibrations}, $\AF$ the \emph{acyclic fibrations}, and $\AC$ the \emph{acyclic cofibrations}.
\end{definition}

Classically, one comes across model categories in the context of the homotopy theory of variants of topological spaces, or algebraic examples such as chain complexes of modules over a ring. Our goal is to focus on the case where the category $\mathcal{C}$ is a finite lattice, which makes the handling of model categories take on a more combinatorial flavor.

\begin{remark}\label{rem:redundancy}
\Cref{def:WFSmodel} contains redundant information, as the definition and the factorization axioms imply the following.
\begin{itemize}
\item $\AC={}^\boxslash \F$ and $\F = \AC^\boxslash$
\item $\C={}^\boxslash \AF$ and $\AF = \C^\boxslash$
\item $\AC = \C \cap \W$ and $\AF = \F \cap \W$
\item $\W = \AF \circ \AC$
\end{itemize}
As a consequence, the entire model category structure is determined by knowing two of the classes, for example, $\W$ and $\AF$, or $\W$ and $\AC$. 
\end{remark}

If all transfer systems on a finite lattice are known, then  we can build all model structures on this lattice as follows. 

\begin{enumerate}
\item For a finite lattice $P$, find all transfer systems $T$ on $P$.
\item Construct all weak factorization systems $({}^\boxslash T, T)$.
\item Determine all pairs of weak factorization systems $({}^\boxslash T_1, T_1)$ and $({}^\boxslash T_2, T_2)$ with $T_2 \subseteq T_1$, or equivalently, ${}^\boxslash T_1 \subseteq {}^\boxslash T_2$. These pairs give us all premodel structures on $P$ with $({}^\boxslash T_1, T_1)=(\AC,\F)$ and $({}^\boxslash T_2, T_2)=(\C,\AF)$.
\item Those premodel structures where $\W=\AF \circ \AC$ satisfies the two-out-of-three axiom are the model structures on $P$.
\end{enumerate}

\begin{example}
Applying this method to $P=[1]$ results in three  model structures. There are ten  model structures on $P=[2]$. These examples are discussed in detail in \cite[Section 4.2]{ModelStructuresOnFiniteTotalOrders}.
\end{example}
    
\begin{example} There are 23 model structures on $P=[1]\times[1]$. We give one below and recommend finding the remaining 22 as an exercise.
    Let $T_2\subseteq T_1$ be the  following transfer systems. 

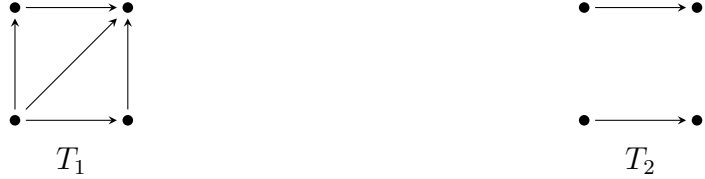
\begin{figure}[h!]
    \centering
    \begin{minipage}{0.45\textwidth}
    \centering
        \begin{tikzpicture}[>=stealth, bend angle=30, baseline=(current bounding box.center)]
    \fill (0,0) circle (2pt);
    \fill (1.5,0) circle (2pt);
    \fill (0,1.5) circle (2pt);
    \fill (1.5,1.5) circle (2pt);
    
    \node (00) at (0,0) {};
    \node (10) at (1.5,0) {};
    \node (01) at (0,1.5) {};
    \node (11) at (1.5,1.5) {};

    \draw[->,white, bend right] (00) to (10);
    \draw[->] (00) --(01);
    \draw[->] (00)--(11);
    \draw[->] (00)--(10);
    \draw[->] (01)--(11);
    \draw[->] (10)--(11);
\end{tikzpicture}

$T_1$
    \end{minipage}
    \begin{minipage}{0.45\textwidth}
    \centering
        \begin{tikzpicture}[>=stealth, bend angle=30, baseline=(current bounding box.center)]
    \fill (0,0) circle (2pt);
    \fill (1.5,0) circle (2pt);
    \fill (0,1.5) circle (2pt);
    \fill (1.5,1.5) circle (2pt);
    
    \node (00) at (0,0) {};
    \node (10) at (1.5,0) {};
    \node (01) at (0,1.5) {};
    \node (11) at (1.5,1.5) {};

    \draw[->] (00)--(10);
    \draw[->] (01)--(11);
    \draw[->,white, bend right] (00) to (10);
\end{tikzpicture}

$T_2$
    \end{minipage}
\end{figure}
We use \Cref{prop:liftformula} to find the  left sets of the weak factorization system.

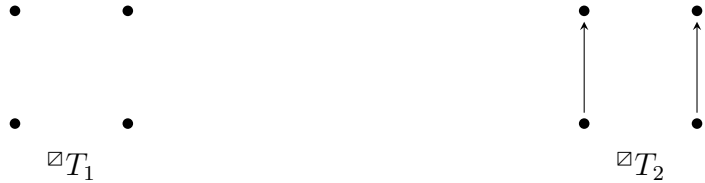
\begin{figure}[h!]
    \centering
    \begin{minipage}{0.45\textwidth}
    \centering
        \begin{tikzpicture}[>=stealth, bend angle=30, baseline=(current bounding box.center)]
    \fill (0,0) circle (2pt);
    \fill (1.5,0) circle (2pt);
    \fill (0,1.5) circle (2pt);
    \fill (1.5,1.5) circle (2pt);
    
    \node (00) at (0,0) {};
    \node (10) at (1.5,0) {};
    \node (01) at (0,1.5) {};
    \node (11) at (1.5,1.5) {};

    \draw[->,white, bend right] (00) to (10);
\end{tikzpicture}

${}^\boxslash T_1$
    \end{minipage}
    \begin{minipage}{0.45\textwidth}
    \centering
        \begin{tikzpicture}[>=stealth, bend angle=30, baseline=(current bounding box.center)]
    \fill (0,0) circle (2pt);
    \fill (1.5,0) circle (2pt);
    \fill (0,1.5) circle (2pt);
    \fill (1.5,1.5) circle (2pt);
    
    \node (00) at (0,0) {};
    \node (10) at (1.5,0) {};
    \node (01) at (0,1.5) {};
    \node (11) at (1.5,1.5) {};

    \draw[->] (00)--(01);
    \draw[->] (10)--(11);
    \draw[->,white, bend right] (00) to (10);
\end{tikzpicture}

${}^\boxslash T_2$
    \end{minipage}
\end{figure}
Since ${}^\boxslash T_1$ is the trivial cotransfer system, the weak equivalence classes correspond to $T_2$. Since $T_2$ satisfies the two-out-of-three axiom, the pairs $({}^\boxslash T_1, T_1)$ and $({}^\boxslash T_2, T_2)$ form a model structure. 
\end{example}

\begin{example}
    Let $P$ be a finite lattice and $T$ a transfer system on $P$. We can pair the weak factorization system $({}^\boxslash T, T)$ with itself to create premodel structure with $T=\AF=\F$ and ${}^\boxslash T=\AC=\C$. By definition of a weak factorization system, every morphism in $P$ can be factored as an acyclic cofibration followed by a fibration, so $\W=\AF \circ \AC=\F \circ \AC = P$. Thus, $\W$ satisfies the 2-out-of-3 property, and the transfer system $T$ together with the weak equivalence set consisting of all morphisms in $P$ defines a model structure on $P$. 
  \end{example}

We use the following proposition extensively in the subsequent sections. 

\begin{lemma}[{\cite[Proposition 1.8]{DrozZakharevichExtending}}]\label{lem:weak=decomp}
Let $P$ be a lattice and $\W$ be the class of weak equivalences of a model structure on $P$. Then $\W$ is decomposable, \emph{i.e.}, if $f=h \circ g \in \W$ then $h \in \W$ and $g \in \W$. 
\end{lemma}

\begin{proof}
Take $f= h \circ g \in \W$. Since $\W=\AF\circ \AC$, we can factor $f$ as an acyclic cofibration followed by an acyclic fibration, $f=f_{af}\circ f_{ac}$. Further, since $(\C,\AF)$ is a weak factorization system, we can factor $g$ as a cofibration followed by an acyclic fibration, $g=g_{af}\circ g_c$. These factorizations fit into the following commutative diagram where $f$ is the morphism from the top left to the bottom right.
\[
\xymatrix{ \bullet \ar[rr]^{f_{ac}} \ar[d]_{g_c} & & \bullet \ar[d]^{f_{af}} \\
\bullet \ar@{.>}[urr]^{k} \ar[r]_{g_{af}} & \bullet \ar[r]_h & \bullet  
}
\]
As we have a cofibration on the left and an acyclic fibration on the right, there is a lift $k$ in the diagram denoted by the dotted arrow. As $P$ is a lattice, $k$ is a pushout of $f_{ac}$ along $g_c$ and, consequently, an acyclic cofibration, making it a weak equivalence. By the two-out-of-three axiom again, $h \in \W$ and therefore also $g \in \W$ as claimed. 
\end{proof}

The following is another key property of model structures on a finite poset that provides information on the short arrows of the poset,  see \Cref{def:short-arrows}.

\begin{remark}\label{rem:shortarrowsareeither} If there is a model structure on a finite lattice $P$ with weak equivalences $\W$, then the factorization property of the weak factorization systems $(\C,\AF)$ and $(\AC,\F)$, and the fact that $\W=\AF\circ \AC$ imply the following about the short arrows of $P$.

\begin{itemize}
\item Every short arrow that is not a weak equivalence has to be both a fibration and a cofibration.
\item Every short arrow that is a weak equivalence has to be a fibration \emph{or} a cofibration, but can never be both.
\end{itemize}
\end{remark}

Thus, every model structure on a finite lattice is determined by a wide decomposable subcategory $\W$ that becomes the class of weak equivalences and a transfer system $T\subseteq\W$ that becomes the class of acyclic fibrations. We  learn more about the properties of such pairs $T \subseteq \W$ in the next two sections. 

\section{From Transfer Systems to Model Structures}\label{sec:FromTStoMS}

In this section, given a finite lattice $P$, we  characterize all model  structures with a given class of weak equivalences $\W$. As discussed in \Cref{rem:redundancy} and \cref{lem:weak=decomp}, any model structure is fully determined by its weak equivalences and acyclic fibrations, with the collection of weak equivalences forming a wide decomposable subcategory and the collection of acyclic fibrations forming a transfer system. In \cref{prop:cotransfer} we give necessary and sufficient conditions for when a wide decomposable subcategory $\W \subseteq P$ and a transfer system $T\subseteq \W$ form the weak equivalences and acyclic fibrations of a model structure. We then use this to describe all model  structures with a given class of weak equivalences  in terms of transfer systems in \Cref{thm:AFW-is-lattice}. Both of these theorems have dual versions where we use that the collection of (acyclic) cofibrations forms a cotransfer system. A different approach is taken in \cite{BMO:composition_closed}, where the authors explore the pairs of transfer systems $T_2 \subseteq T_1$ such that the potential weak equivalences $\W=T_2 \circ ({}^\boxslash T_1)$ satisfy the weaker condition of being closed under composition, rather than being closed under two-out-of-three.

For our first example, consider the poset $[n]$ for some $n\geq0$. The wide decomposable subcategories of $[n]$ correspond to partitions, and as shown in \cite{ModelStructuresOnFiniteTotalOrders}, every pair of wide decomposable subcategory $\W$ along with a transfer system $T \subseteq \W$ defines a model structure on $[n]$. However, it is not always as simple, as the next two examples show.

\begin{example}\label{ex:square}
Consider the wide decomposable subcategory $\W$ in $[1]\times [1]$ shown below, and let $T$ be the transfer system consisting of only the bottom horizontal arrow.
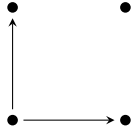
\[
\begin{tikzpicture}[>=stealth, bend angle=20, baseline=(current bounding box.center)]
    \fill (0,0) circle (2pt);
    \fill (1.5,0) circle (2pt);
    \fill (0,1.5) circle (2pt);
    \fill (1.5,1.5) circle (2pt);
    
    \node (00) at (0,0) {};
    \node (10) at (1.5,0) {};
    \node (01) at (0,1.5) {};
    \node (11) at (1.5,1.5) {};

    \draw[->] (00) -- (01);
    \draw[->] (00)--(10);
    \draw[->, white, bend left] (01) to (11);
    \draw[->, white, bend right] (00) to (10);

    \draw[lightgray] (-0.5,-0.5)--(-0.5,2)--(2,2)--(2,-0.5)--(-0.5,-0.5);

    \node (W) at (0.75,-0.8){$\W$};
\end{tikzpicture}
\]
 We argue that there is no model structure with weak equivalences $\W$ and this particular choice of $\AF=T$. If there were, then the cofibrations would be given by $\C={}^{\boxslash}T=\mathcal{E}_d(T)^c$, which is the following.
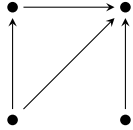
\[
\begin{tikzpicture}[>=stealth, bend angle=20, baseline=(current bounding box.center)]
    \fill (0,0) circle (2pt);
    \fill (1.5,0) circle (2pt);
    \fill (0,1.5) circle (2pt);
    \fill (1.5,1.5) circle (2pt);
    
    \node (00) at (0,0) {};
    \node (10) at (1.5,0) {};
    \node (01) at (0,1.5) {};
    \node (11) at (1.5,1.5) {};

    \draw[->] (00) -- (01);
    \draw[->] (10)--(11);
    \draw[->] (01)--(11);
    \draw[->] (00)--(11);
    \draw[->, white, bend left] (01) to (11);
    \draw[->, white, bend right] (00) to (10);

    \draw[lightgray] (-0.5,-0.5)--(-0.5,2)--(2,2)--(2,-0.5)--(-0.5,-0.5);

    \node (W) at (0.75,-0.8){$\C$};
\end{tikzpicture}
\]
Since $\AC=\C\cap\W$,  the acyclic cofibrations would consist only of the left vertical arrow in $\W$, which is not a cotransfer system because it is not closed under pushouts. Hence, there is no model structure with the chosen classes.
\end{example}

It is possible to build a model structure whose weak equivalences are given by the $\W$ of \Cref{ex:square}, just not one with $\AF$ given by the $T$ given above. However, on larger lattices, it is possible to find a wide decomposable subcategory that never arises as the class of weak equivalences of a model structure.

\begin{example}\label{ex:middle-arrow}
    Consider the lattice $[2]\times[1]$ and the wide decomposable subcategory $Q$ shown below.
     \[
\begin{tikzpicture}[>=stealth, bend angle=20, baseline=(current bounding box.center)]
    \fill (0,0) circle (2pt);
    \fill (1.5,0) circle (2pt);
    \fill (3,0) circle (2pt);
    \fill (0,1.5) circle (2pt);
    \fill (1.5,1.5) circle (2pt);
    \fill (3,1.5) circle (2pt);
    
    \node (00) at (0,0) {};
    \node (10) at (1.5,0) {};
    \node (20) at (3,0) {};
    \node (01) at (0,1.5) {};
    \node (11) at (1.5,1.5) {};
    \node (21) at (3,1.5) {};

    \draw[->] (10)--(11);
    \draw[->, white, bend left] (01) to (21);
    \draw[->, white, bend right] (00) to (20);
\end{tikzpicture}
\]
By \Cref{rem:shortarrowsareeither}, if $Q$ were a class of weak equivalences, then the one nontrivial arrow in $Q$ must be either an acyclic fibration or an acyclic cofibration, so either $\AF=Q$ or $\AC=Q$. Neither is possible since $Q$ is neither closed under pushouts nor under pullbacks, meaning it is neither a cotransfer system nor a transfer system. Thus, there are no  model structures with this particular $Q$ as the class of weak equivalences. We revisit this case, along with other wide decomposable subcategories of $[2]\times [1]$, in \Cref{ex:burgerfailnocenter}.
\end{example}

For a general finite lattice $P$, given a wide decomposable subcategory $Q$ and a transfer system $T \subseteq Q$, we  need tools to determine if there is a model  structure with weak equivalences $\W=Q$ and acyclic fibrations $\AF = T$. \Cref{rem:redundancy} motivates the following notation.

\begin{notation}\label{notation}
    We will use the subscript ``temp" to indicate that the collection is a candidate for that particular class of morphisms in a model structure. Given a finite lattice $P$, a wide decomposable subcategory $Q$ and a transfer system $T \subseteq Q$, define
    \begin{itemize}
    \item $\Wtemp \coloneqq Q$,
    \item $\AFtemp \coloneqq T$,
    \item $\Ctemp \coloneqq {}^{\boxslash}\AFtemp$,
    \item $\ACtemp \coloneqq \Ctemp \cap \Wtemp$, and
    \item $\Ftemp \coloneqq \ACtemp{}^\boxslash$.
    \end{itemize}
\end{notation}

We  now show that the collection of classes of morphisms described in \Cref{notation} satisfies the factorization axiom for morphisms in a weak equivalence class. 

\begin{lemma}\label{lem:wtemp} Let $P$ be a finite lattice, let $Q\subseteq P$ be a wide decomposable subcategory, and let $T \subseteq Q$ be a transfer system. Let $\Wtemp$, $\AFtemp$ and $\ACtemp$ be as in \Cref{notation}. Then
\[
\Wtemp=\AFtemp \circ \ACtemp.
\]
\end{lemma}

\begin{proof} First, $\AFtemp =T\subseteq Q=\Wtemp$ and $\ACtemp = \Ctemp \cap \Wtemp$. Hence,  both $\AFtemp$ and $\ACtemp$ are contained in $\Wtemp$. Since $\Wtemp$ is closed under composition, it follows that $\AFtemp \circ \ACtemp \subseteq \Wtemp$.

For the converse, since $\AFtemp$ is a transfer system, by \Cref{rem:halftransfer} $(\Ctemp,\AFtemp)$ is a weak factorization system. Thus, if $f \in \Wtemp$, then we can factor  $f$ as $f = g \circ h$ for $g \in \AFtemp$ and $h \in \Ctemp$. As $\Wtemp=Q$ is decomposable by assumption, we have that $h$ and $g$ must also be in $\Wtemp$, so $h \in \ACtemp$.
\end{proof}

With this result in hand, we   give a necessary and sufficient condition for when a wide decomposable subcategory and a transfer system contained in it gives the weak equivalences and acyclic fibrations for a model structure.

\begin{theorem}\label{prop:cotransfer}
    Let $P$ be a finite lattice and let  $\W \subseteq P$ be a wide decomposable subcategory. If $T \subseteq \W$ is a transfer system, then there is a model structure in which $\W$ is the set of weak equivalences and $T$ the acyclic fibrations if and only if ${}^\boxslash T \cap \W$ is a cotransfer system on $P$. 
    
    Dually, if $K \subseteq \W$ is a cotransfer system, then there is a model structure on $P$ with weak equivalences $\W$ and acyclic cofibrations $K$ if and only if $K^\boxslash \cap \W$ is a transfer system.
\end{theorem}

\begin{proof}
First, if there is a model structure with weak equivalences $\W$ and $\AF = T$, then by \Cref{def:WFSmodel}, $\AC = {}^\boxslash T \cap \W$ is the left set of a weak factorization system. Hence, ${}^\boxslash T \cap \W$ is a cotransfer system.

Conversely, assume ${}^\boxslash T \cap \W$ is a cotransfer system.  Let $\AFtemp = T$, and $\Ctemp$, $\ACtemp$, and $\Ftemp$ be as defined in \Cref{notation}. In particular, $\Ctemp ={}^{\boxslash}T$ and $\ACtemp = {}^{\boxslash}T\cap \W$. We will show that this collection forms a model structure. Since by assumption $\AFtemp$ is a transfer system and $\ACtemp$ is a cotransfer system, we have two weak factorization systems
\[
(\Ctemp,\AFtemp) \text{ and } (\ACtemp,\Ftemp) \text{ with } \ACtemp \subseteq \Ctemp, \text{ or equivalently } \AFtemp \subseteq \Ftemp.\]
Thus, this pair is a premodel structure. By  \Cref{lem:wtemp}, $\AFtemp \circ \ACtemp=
\W$, with $\W$ satisfying the 2-out-of-3 property. Therefore, this collection gives a model structure and not just a premodel structure. 

The second statement has a dual proof.
\end{proof}

 \Cref{prop:cotransfer} allows us to determine with relative ease whether a given  wide decomposable subcategory paired with a (co)transfer system  leads to a model structure.

\begin{example}\label{ex:nottriv}
Consider the lattice $P = [2]\times[1]$. Let $Q$ be the wide decomposable subcategory of $P$ shown below.
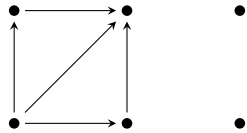
\[
\begin{tikzpicture}[>=stealth, bend angle=20, baseline=(current bounding box.center)]
    \fill (0,0) circle (2pt);
    \fill (1.5,0) circle (2pt);
    \fill (3,0) circle (2pt);
    \fill (0,1.5) circle (2pt);
    \fill (1.5,1.5) circle (2pt);
    \fill (3,1.5) circle (2pt);
    
    \node (00) at (0,0) {};
    \node (10) at (1.5,0) {};
    \node (20) at (3,0) {};
    \node (01) at (0,1.5) {};
    \node (11) at (1.5,1.5) {};
    \node (21) at (3,1.5) {};

    \draw[->] (00) -- (01);
    \draw[->] (00)--(10);
    \draw[->] (00)--(11);
    \draw[->] (10)--(11);
    \draw[->] (01)--(11);
    \draw[->, white, bend left] (01) to (21);
    \draw[->, white, bend right] (00) to (20);
\end{tikzpicture}
\]
We use \Cref{prop:cotransfer} to show that there is no model structure with $\W=Q$ and $\AF=T_{triv}$, the trivial transfer system consisting only of identity morphisms. Since ${}^\boxslash T_{triv}=P$ it follows that ${}^\boxslash T_{triv}\cap Q = Q$. However, $Q$ is not closed under pushouts and therefore is not a cotransfer system. By \Cref{prop:cotransfer}, there is no model structure with $\W=Q$ and $\AF=T_{triv}$.
\end{example}

\begin{example}\label{ex:ktriv}
Let $P=[2]\times[1]$, $Q$ be the decomposable subcategory depicted in \Cref{ex:nottriv}.  \Cref{prop:cotransfer} implies that there is a model structure with $\W=Q$ and $\AC=K_{triv}$, the trivial cotransfer system, since $K_{triv}^\boxslash \cap Q = P \cap Q = Q$ is a transfer system.
\end{example}

These examples generalize to the following result. 

\begin{proposition}\label{prop:cotransferW}
    Let $\W$ be a wide decomposable subcategory of a finite lattice $P$. Then there exists a model structure with weak equivalences given by $\W$ and $\AF=T_{triv}$   if and only if $\W$ is a cotransfer system on $P$.
    
    Dually, there exists a model structure with $\AC=K_{triv}$ and weak equivalences given by $\W$ if and only if $\W$ is a transfer system on $P$.
\end{proposition}

\begin{proof}
    Assume there exists a model structure with weak equivalences given by $\W$ and acyclic fibrations given by $T_{triv}$. As $\W=\AF \circ \AC$, this means that $\W=\AC$, which in turn implies that $\W$ is a cotransfer system. 
    
    Conversely, suppose $\W$ is a cotransfer system, and consider the trivial transfer system $T_{triv} \subseteq \W$. Then $^\boxslash T_{triv}=P$, and $^\boxslash T_{triv}\cap \W = \W$. \Cref{prop:cotransfer} implies that there is a model structure with weak equivalences given by $\W$ and acyclic fibrations given by $T_{triv}$, as wanted.

    The second statement follows by a dual argument.
\end{proof}

\begin{remark}
    Recall from \cref{defn:saturated} that a saturated transfer system on a finite lattice is both a wide decomposable subcategory and a transfer system. Thus, \Cref{prop:cotransferW} implies that if $T$ is a saturated transfer system then there exists a model structure with $\W=\AF=T$ and $\AC=K_{triv}$.
\end{remark}

The following definition allows us to speak more concisely about wide decomposable subcategories that give rise to a model  structure.

\begin{definition}\label{def:weak-eq-set} Let $\W$ be a wide decomposable subcategory of a finite lattice $P$. We say that $\W$ is a \emph{weak equivalence set} if there exists a model category structure on $P$ with weak equivalences $\W$.
\end{definition}

\begin{example}\label{ex:cotransferweakequivalence}
By \Cref{prop:cotransferW}, if a wide decomposable subcategory is a transfer system or a cotransfer system, then it is  a weak equivalence set.
\end{example}

For the rest of this section we  study all model structures on a finite lattice $P$ with a given weak equivalence set $\W$. We  prove that this collection is a sublattice of the lattice of transfer systems on $P$. In particular, this shows that there exists a minimal and a maximal model structure with weak equivalences $\W$ when ordering by containment of the acyclic fibrations.

\begin{definition}[$\AFW$ and $\ACW$]
    Given a finite lattice $P$ and weak equivalence set $\W$, we let $\AFW$ be the set consisting of those transfer systems contained in $\W$ that form the acyclic fibrations of a model structure on $P$ with weak equivalences  $\W$. We similarly let $\ACW$ be the set consisting of those cotransfer systems contained in $\W$ that form the acyclic cofibrations of a model structure on $P$ with weak equivalences $\W$.
\end{definition}

Recall that in the lattice of transfer systems on $P$ the meet $T_1 \wedge T_2$ is given by the intersection $T_1 \cap T_2$, and the join $T_1 \vee T_2$ is given by the transfer system generated by the union $T_1 \cup T_2$, which can be obtained by closing $T_1\cup T_2$ under compositions as seen in \Cref{lem: joins}. There are similar descriptions for the meet and the join of cotransfer systems.

\begin{proposition}\label{prop:intersection_of_AF}
  Let $\W$ be a weak equivalence set on a finite lattice $P$. If $T_1,T_2\in \AFW$, then $T_1\wedge T_2 \in \AFW$. Dually, if $K_1,K_2\in \ACW$, then $K_1 \wedge K_2 \in \ACW$.
\end{proposition}

\begin{proof}
    By \Cref{prop:cotransfer}, it suffices to show that $^\boxslash (T_1 \wedge T_2)\cap \W$ is a cotransfer system. Since both $^\boxslash (T_1 \wedge T_2)$ and $\W$ are closed under composition and contain all identities, their intersection likewise is closed under composition and contains all identities. Thus, we must show that $^\boxslash (T_1 \wedge T_2)\cap \W$ is closed under pushouts. 
    The order-reversing bijection in \Cref{rem:halftransfer} implies that $^\boxslash (T_1 \wedge T_2)$ is the cotransfer system given by $(^\boxslash T_1) \vee (^\boxslash T_2)$, which as noted above is  the closure under composition of $(^\boxslash T_1) \cup (^\boxslash T_2)$.

    Now take $f\in {}^\boxslash(T_1 \wedge T_2)\cap \W$. Then $f=g_k \circ \dots \circ g_1$, where each $g_i$ is in $^\boxslash T_1$ or $^\boxslash T_2$. Any pushout of $f$ is a composite of pushouts of the $g_i$, so it is enough to prove that the pushout of each $g_i$ is in $^\boxslash (T_1 \wedge T_2)\cap \W$. Fix $i$ and, without loss of generality, assume $g_i \in {}^\boxslash T_1$. Since $\W$ is decomposable, $g_i$ is in $\W$, and thus is in $(^\boxslash T_1) \cap \W$. Our assumption that $T_1\in \AFW$ together with \Cref{prop:cotransfer} implies that $(^\boxslash T_1) \cap \W$ is a  cotransfer system, and hence closed under pushouts. Thus, any pushout of $g_i$ is in $(^\boxslash T_1) \cap \W \subseteq {}^\boxslash (T_1 \wedge T_2)\cap \W$, finishing the proof.

    A dual proof gives the second statement.
\end{proof}

\begin{remark}
    \Cref{prop:intersection_of_AF} is a special case of a general result about model structures for arbitrary categories, found for example in \cite[Proposition 4.4.11]{BalchinBook}: Given a bicomplete category $\mathcal{C}$ with two model structures with the same collection of weak equivalences $\W$, and acyclic fibrations $\AF_1$ and $\AF_2$ respectively,  there exists a model structure on $\mathcal{C}$ with weak equivalences $\W$ and acyclic fibrations given by $\AF_1 \cap \AF_2$. 
\end{remark}

The fact that the meet of two transfer systems in $\AFW$ is again in $\AFW$ implies the existence of a minimal transfer system in $\AFW$. Similarly, there is a minimal cotransfer system in $\ACW$.

\begin{definition}[$\AF_{min}$ and $\AC_{min}$]
Let $\W$ be a weak equivalence set on a finite lattice $P$. Define $\AF_{min}$ in $\W$ to be the meet of all transfer systems in $\AFW$. Similarly, define $\AC_{min}$ to be the meet of all cotransfer systems in $\ACW$.
\end{definition}

\begin{example}
Consider the lattice $P = [1] \times [1]$. Since the following wide decomposable subcategory $\W$ is a transfer system, it is a weak equivalence set on $P$. 
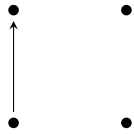
\[
\begin{tikzpicture}[>=stealth, bend angle=20, baseline=(current bounding box.center)]
    \fill (0,0) circle (2pt);
    \fill (1.5,0) circle (2pt);
    \fill (0,1.5) circle (2pt);
    \fill (1.5,1.5) circle (2pt);
    
    \node (00) at (0,0) {};
    \node (10) at (1.5,0) {};
    \node (01) at (0,1.5) {};
    \node (11) at (1.5,1.5) {};

    \draw[->] (00) -- (01);
    \draw[->, white, bend right] (00) to (10);
    \draw[->, white, bend left] (01) to (11);
    \draw[lightgray] (-0.5,-0.5)--(-0.5,2)--(2,2)--(2,-0.5)--(-0.5,-0.5);

    \node (W) at (0.75,-0.8){$\W$};
\end{tikzpicture}
\]
There are two transfer systems contained in $\W$, the trivial transfer system and $\W$. By \Cref{prop:cotransferW}, since $\W$ is not a cotransfer system, $T_{triv}$ is not an element of $\AFW$. Hence $\AF_{min}=\W$.
\end{example}

We now focus on showing that $\AFW$ and $\ACW$ are closed under joins, and thus each has a maximal element. We begin with a preliminary result.

\begin{proposition}\label{prop:range} Let $\W$ be a weak equivalence set on a finite lattice $P$. If $T \in \AFW$ and $T'$ is a transfer system such that $T \subseteq T' \subseteq \W$, then $T'\in \AFW$.

Dually,  if $K \in \ACW$ and $K'$ is a cotransfer system such that $K \subseteq K' \subseteq \W$, then $K'\in \ACW$.
\end{proposition}

\begin{proof} By assumption, there is a model category structure with weak equivalences $\W$, acyclic fibrations $\AF=T$, cofibrations $\C={}^\boxslash T$ and the acyclic cofibrations $\AC={}^\boxslash T \cap \W$. 

For a transfer system $T'$ such that $T \subseteq T' \subseteq \W$, we define $\mathsf{AF'}=T'$, $\mathsf{C'}={}^\boxslash T'$, and $\mathsf{AC'} = \mathsf{C'} \cap \W$. By \Cref{prop:cotransfer}, it suffices to prove that $\mathsf{AC'}$ is a cotransfer system. We know that $\C$, $\mathsf{C'}$, and $\AC,$ are cotransfer systems, and since $T \subseteq T'$, it follows that $\C \supseteq \mathsf{C'}.$
    We will show that $\mathsf{AC'} = \mathsf{C'} \cap \AC$. Then since both $\mathsf{C'}$ and $\AC$ are cotransfer systems, so is $\mathsf{AC'}$. 
    
    In one direction, since $\mathsf{C'} \subseteq \C$, we have that $\mathsf{AC'} = \mathsf{C'} \cap \W \subseteq \C \cap \W = \AC$. Therefore $\mathsf{AC'} \subseteq \mathsf{C'} \cap \AC$. In the other direction, by definition $\AC \subseteq \W$ so $\mathsf{C'} \cap \AC \subseteq \mathsf{C'} \cap \W = \mathsf{AC'}$. We conclude that $\mathsf{AC'} = \mathsf{C'} \cap \AC$. 
\end{proof}

We can now prove that $\AFW$ and $\ACW$ are closed under joins and thus contain maximal elements.

\begin{proposition}\label{prop:union_of_AF}
  Let $\W$ be a weak equivalence set on a finite lattice $P$. If $T_1$ and $T_2$ are transfer systems in $\AFW$, then $T_1 \vee T_2 \in \AFW$. 
  
  Dually, if $K_1$ and $K_2$ and cotransfer systems in $\ACW$, then $K_1 \vee K_2 \in \ACW$.
\end{proposition}

\begin{proof}
    By \Cref{lem:vee-in-W}, $T_1\vee T_2$ is a transfer system satisfying $T_1 \subseteq T_1 \vee T_2 \subseteq \W$, and thus, by \Cref{prop:range}, we have that $T_1 \vee T_2 \in \AFW$.
    The dual proof is similar. \end{proof}

\begin{definition}[$\AF_{max}$ and $\AC_{max}$]\label{def:AFmax and ACmax}
    Let $\W$ be a weak equivalence set on a finite lattice $P$. Define $\AF_{max}$ to be the join of all transfer systems in $\AFW$ and $\AC_{max}$ to be the join of all cotransfer systems in $\ACW$.
\end{definition}

The above results combine to give the following theorem and its dual.

\begin{theorem}\label{thm:AFW-is-lattice} Let $\W$ be a weak equivalence set on a finite lattice $P$. Then 
\begin{enumerate}
    \item $\AFW$ is a lattice with meet and join inherited from $\Tr(P)$, and
    \item in particular, $\AFW$ is the interval sublattice $$[\AF_{min},\AF_{max}] =\{ T \in \Tr(P) \mid \AF_{min} \subseteq T \subseteq \AF_{max} \}.$$
\end{enumerate}

\end{theorem}

\begin{proof}
    Part (1) follows directly from \Cref{prop:intersection_of_AF} and \Cref{prop:union_of_AF}. Part (2) follows from the definitions of $\AF_{min}$ and $\AF_{max}$ and \Cref{prop:range}.
\end{proof}

\begin{theorem}\label{thm:ACW-is-lattice}
    Let $\W$ be a weak equivalence set on a finite lattice $P$. Then 
    \begin{enumerate}
    \item $\ACW$ is a lattice with meet and join inherited from $\coTr(P)$, and
    \item in particular, $\ACW$ is the interval sublattice $$[\AC_{min},\AC_{max}] =\{ K \in \coTr(P) \mid \AC_{min} \subseteq K \subseteq \AC_{max} \}.$$
    \end{enumerate}
\end{theorem}

\begin{example}
We can use \Cref{thm:AFW-is-lattice} to recover the fact that every pair of partition $\W$ on $[n]$ with transfer system $T \subseteq \W$, defines a model structure on $[n]$. Since every wide decomposable subcategory $\W$ on $[n]$ is a partition, each $\W$ is automatically closed under pullbacks and pushouts and thus is both a transfer system and cotransfer system. By \Cref{prop:cotransferW}, $\W$ is a weak equivalence set and there exists a model structure with weak equivalences $\W$ and $\AF=T_{triv}$, and separately, a model structure with weak equivalences $\W$ and $\AC=K_{triv}$, and so $\AF=\W$. Thus, $\AF_{min} = T_{triv}$, $\AF_{max} = \W$, and $\AFW = [T_{triv}, W]$.

\end{example}

By \Cref{lem:vee-in-W} every weak equivalence set $\W$ of a finite lattice contains a maximum transfer system $\Tmax$, defined as the join of all transfer systems contained in $\W$. The transfer system $\AF_{max}$, on the other hand, is the maximum transfer system in the sublattice $\AFW$. By \Cref{prop:range}, these two maximal transfer systems must be equal. A similar argument gives the dual statement thus proving the following statement. 

\begin{proposition}\label{prop:AFmax=Tmax} Let $\W$ be a weak equivalence set on a finite lattice $P$. Then $\AF_{max} = \Tmax$ and $\AC_{max}=\Cmax$.
\end{proposition}

\begin{remark}\label{rem:Wcotransfer} Given a wide decomposable subcategory $\W$ of a finite lattice $P$, by \Cref{prop:cotransferW}, $\W$ is a weak equivalence set with $\AF_{min}=T_{triv}$ if and only if $\W$ is a cotransfer system. Moreover, if $\AF_{min}=T_{triv}$, since $\AF_{max} = T_{max}$, by \Cref{thm:AFW-is-lattice}, it follows that \textit{every} transfer system in $\W$ defines a model structure. Thus, if $\W$ is a cotransfer system then every transfer system contained in $\W$ defines a model structure. Dually, if $\W$ is a transfer system then every cotransfer system in $\W$ defines a model structure.  

\end{remark}

 Since a model structure is uniquely determined either by  the data of $\W$ and $\AF$ or by the data of $\W$ and $\AC$, see \Cref{rem:redundancy}, there is a bijection between $\AFW$ and $\ACW$ sending the acyclic fibrations of a model structure to the acyclic cofibrations of the same model structure. This bijection gives a duality between $\AFW$ and $\ACW$.

\begin{theorem}\label{thm:duality-AF-AC}
    Let $\W$ be a weak equivalence set on a finite lattice $P$. Then there is an order-reversing bijection $\AFW \to \ACW$ given by the assignment
    \[
    T \longmapsto {}^\boxslash T \cap \W.
    \]
\end{theorem}

\begin{proof}
    This map sends the acyclic fibrations of a model structure to the corresponding acyclic cofibrations. The fact that model structures are uniquely determined by the pair of weak equivalences and acyclic fibrations, or by the pair of weak equivalences and acyclic cofibrations, together with \Cref{prop:cotransfer} give the statement about the bijection. To prove that this is order-reversing, note that if $T \subseteq T'$, then $^ \boxslash T' \subseteq {}^ \boxslash T$, and hence  $^ \boxslash T'\cap \W  \subseteq {}^ \boxslash T\cap \W$.
\end{proof}

Since the bijection of \Cref{thm:duality-AF-AC} is order reversing, it follows that $\AF_{max} \mapsto \AC_{min}$ and $\AF_{min} \mapsto \AC_{max}$. This, together with \Cref{prop:AFmax=Tmax}, gives us an effective way to compute $\AF_{min}$, and dually $\AC_{min}$.

\begin{corollary}\label{cor:dual-AFmin-ACmin} Let $\W$ be a weak equivalence set on a finite lattice $P$. Then
$$\AF_{min} = {K_{max}}^\boxslash \cap \W  \hspace{4pt}\text{ and }\hspace{4pt} \AC_{min} = {}^\boxslash T_{max} \cap \W.$$
Similarly, $$T_{max} = {\AC_{min}}^\boxslash \cap \W\hspace{4pt}\text{ and }\hspace{4pt} K_{max} = {}^\boxslash \AF_{min} \cap \W.$$
\end{corollary}

We end this section by discussing how we can determine the sublattice $\AFW$ for a given weak equivalence set $\W$. We  determine $\ACW$ using a dual procedure. Since $\AFW$ is an interval lattice, it is determined by $\AF_{min}$ and $\AF_{max}$. By \Cref{prop:AFmax=Tmax}, $\AF_{max}$ is the maximum transfer system  in $\W$, which can be constructed explicitly using \cref{prop:Tmax-closure}. Given a weak equivalence set $\W$, we  directly calculate $\AF_{min}$  using the following steps.

\begin{enumerate}
\item Use \cref{prop:Kmax-closure} to find the largest cotransfer system $K_{max}$ that fits inside of $\W$. 
\item Use \Cref{prop:liftformula} to calculate ${K_{max}}^\boxslash=\mathcal{E}_u(K_{max})^c$.
\item Calculate $\AF_{min}:={K_{max}}^\boxslash \cap \W$.
\end{enumerate}
We dually calculate $\AC_{min}$ directly in a similar way. The following example illustrates this process.

\begin{example} 
    Consider the following wide decomposable category $\W$ on $[2]\times[2]$ with  maximum cotransfer system $\Cmax$ and transfer system $\Tmax$, which are calculated using \cref{prop:Kmax-closure} and \cref{prop:Tmax-closure}, respectively.

\[
    \begin{minipage}{0.3\textwidth}
    \centering
        \begin{tikzpicture}[>=stealth, bend angle=20, baseline=(current bounding box.center)]
    \fill (0,0) circle (2pt);
    \fill (1.5,0) circle (2pt);
    \fill (3,0) circle (2pt);
    \fill (0,1.5) circle (2pt);
    \fill (1.5,1.5) circle (2pt);
    \fill (3,1.5) circle (2pt);
    \fill (0,3) circle (2pt);
    \fill (1.5,3) circle (2pt);
    \fill (3,3) circle (2pt);
    
    \node (00) at (0,0) {};
    \node (10) at (1.5,0) {};
    \node (20) at (3,0) {};
    \node (01) at (0,1.5) {};
    \node (11) at (1.5,1.5) {};
    \node (21) at (3,1.5) {};
    \node (02) at (0,3) {};
    \node (12) at (1.5,3) {};
    \node (22) at (3,3) {};

    \draw[->] (00) -- (01);
    \draw[->] (01) --(02);
    \draw[->] (11)--(12);
    \draw[->] (20)--(21);
    \draw[->] (21)--(22);
    \draw[->, bend left] (00) to (02);
    \draw[->, bend right] (20) to (22);

    \draw[lightgray] (-0.5,-0.5)--(-0.5,3.5)--(3.5,3.5)--(3.5,-0.5)--(-0.5,-0.5);

    \node (W) at (1.5,-0.8){$\W$};
\end{tikzpicture}

    \end{minipage}
    \begin{minipage}{0.3\textwidth}
    \centering
        \begin{tikzpicture}[>=stealth, bend angle=20, baseline=(current bounding box.center)]
    \fill (0,0) circle (2pt);
    \fill (1.5,0) circle (2pt);
    \fill (3,0) circle (2pt);
    \fill (0,1.5) circle (2pt);
    \fill (1.5,1.5) circle (2pt);
    \fill (3,1.5) circle (2pt);
    \fill (0,3) circle (2pt);
    \fill (1.5,3) circle (2pt);
    \fill (3,3) circle (2pt);
    
    \node (00) at (0,0) {};
    \node (10) at (1.5,0) {};
    \node (20) at (3,0) {};
    \node (01) at (0,1.5) {};
    \node (11) at (1.5,1.5) {};
    \node (21) at (3,1.5) {};
    \node (02) at (0,3) {};
    \node (12) at (1.5,3) {};
    \node (22) at (3,3) {};

    \draw[->] (01) --(02);
    \draw[->] (11)--(12);
    \draw[->] (20)--(21);
    \draw[->] (21)--(22);
    \draw[->, bend right] (20) to (22);

    \draw[lightgray] (-0.5,-0.5)--(-0.5,3.5)--(3.5,3.5)--(3.5,-0.5)--(-0.5,-0.5);

    \node (Kmax) at (1.5,-0.8){$\Cmax$};
\end{tikzpicture}
    \end{minipage}
    \begin{minipage}{0.3\textwidth}
    \centering
        \begin{tikzpicture}[>=stealth, bend angle=20, baseline=(current bounding box.center)]
    \fill (0,0) circle (2pt);
    \fill (1.5,0) circle (2pt);
    \fill (3,0) circle (2pt);
    \fill (0,1.5) circle (2pt);
    \fill (1.5,1.5) circle (2pt);
    \fill (3,1.5) circle (2pt);
    \fill (0,3) circle (2pt);
    \fill (1.5,3) circle (2pt);
    \fill (3,3) circle (2pt);
    
    \node (00) at (0,0) {};
    \node (10) at (1.5,0) {};
    \node (20) at (3,0) {};
    \node (01) at (0,1.5) {};
    \node (11) at (1.5,1.5) {};
    \node (21) at (3,1.5) {};
    \node (02) at (0,3) {};
    \node (12) at (1.5,3) {};
    \node (22) at (3,3) {};

    \draw[->] (01) --(02);
    \draw[->] (11)--(12);
    \draw[->] (00)--(01);
    \draw[->] (21)--(22);
    \draw[->, bend left] (00) to (02);

    \draw[lightgray] (-0.5,-0.5)--(-0.5,3.5)--(3.5,3.5)--(3.5,-0.5)--(-0.5,-0.5);

    \node (Tmax) at (1.5,-0.8){$\Tmax$};
\end{tikzpicture}

    \end{minipage}
\]
We first use \cref{prop:liftformula} to compute $\Cmax^\boxslash$, from which we obtain $\Cmax^\boxslash\cap\W$, both of which are shown below. Since $\Cmax{^\boxslash}\cap \W$ is a transfer system, \Cref{prop:cotransfer}, implies that $\W$ is a weak equivalence set. Additionally, by \Cref{cor:dual-AFmin-ACmin}, we get $\AF_{min} = \Cmax{^\boxslash}\cap \W$.
\[
\begin{minipage}{0.35\textwidth}
\centering
\begin{tikzpicture}[>=stealth, bend angle=20, baseline=(current bounding box.center)]
    \fill (0,0) circle (2pt);
    \fill (1.5,0) circle (2pt);
    \fill (3,0) circle (2pt);
    \fill (0,1.5) circle (2pt);
    \fill (1.5,1.5) circle (2pt);
    \fill (3,1.5) circle (2pt);
    \fill (0,3) circle (2pt);
    \fill (1.5,3) circle (2pt);
    \fill (3,3) circle (2pt);
    
    \node (00) at (0,0) {};
    \node (10) at (1.5,0) {};
    \node (20) at (3,0) {};
    \node (01) at (0,1.5) {};
    \node (11) at (1.5,1.5) {};
    \node (21) at (3,1.5) {};
    \node (02) at (0,3) {};
    \node (12) at (1.5,3) {};
    \node (22) at (3,3) {};

    \draw[->] (00)--(01);
    \draw[->, bend left] (00) to (02);
    \draw[->] (00)--(12);
    \draw[->] (00)--(11);
    \draw[->] (00)--(10);
    \draw[->, bend left] (00) to (22);
    \draw[->] (00)--(21);
    \draw[->, bend right] (00) to (20);
    \draw[white, line width=4pt] (01)--(11);
    \draw[->] (01)--(11);
    \draw[->, bend left] (01) to (21);
    \draw[->] (02)--(12);
    \draw[->, bend left] (02) to (22);
    \draw[white, line width=4pt] (10)--(11);
    \draw[->] (10)--(11);
    \draw[->] (10)--(20);    
    \draw[->, bend right] (10) to (12);    
    \draw[->] (10)--(22);
    \draw[->] (10)--(21);
    \draw[white, line width=4pt] (11)--(21);
    \draw[->] (11)--(21);
    \draw[->] (12)--(22);

    \draw[lightgray] (-0.5,-0.5)--(-0.5,3.5)--(3.5,3.5)--(3.5,-0.5)--(-0.5,-0.5);

    \node (Kmax) at (1.5,-0.8){$\Cmax^\boxslash$};
\end{tikzpicture}
\end{minipage}
\begin{minipage}{0.35\textwidth}
\centering
\begin{tikzpicture}[>=stealth, bend angle=20, baseline=(current bounding box.center)]
    \fill (0,0) circle (2pt);
    \fill (1.5,0) circle (2pt);
    \fill (3,0) circle (2pt);
    \fill (0,1.5) circle (2pt);
    \fill (1.5,1.5) circle (2pt);
    \fill (3,1.5) circle (2pt);
    \fill (0,3) circle (2pt);
    \fill (1.5,3) circle (2pt);
    \fill (3,3) circle (2pt);
    
    \node (00) at (0,0) {};
    \node (10) at (1.5,0) {};
    \node (20) at (3,0) {};
    \node (01) at (0,1.5) {};
    \node (11) at (1.5,1.5) {};
    \node (21) at (3,1.5) {};
    \node (02) at (0,3) {};
    \node (12) at (1.5,3) {};
    \node (22) at (3,3) {};

    \draw[->] (00)--(01);
    \draw[->, bend left] (00) to (02);

    \draw[lightgray] (-0.5,-0.5)--(-0.5,3.5)--(3.5,3.5)--(3.5,-0.5)--(-0.5,-0.5);

    \node (Kmax) at (1.5,-0.8){$\Cmax^\boxslash\cap \W$};
\end{tikzpicture}

\end{minipage}
\] 
We can further determine all model structures with weak equivalences $\W$. We know that $\AF_{max} = \Tmax$. Thus, $\AFW$ consists of the four transfer systems $T$ such that $\AF_{min} \le T \le \AF_{max}$, which are constructed by successively adding the top  short vertical arrows of $\W$ from left to right.
\end{example}

\section{Weak Equivalence Sets on Finite Lattices}\label{sec:centers}

The goal of this section is to give necessary and sufficient conditions for a wide decomposable subcategory $Q$ of a finite lattice $P$ to be the class of weak equivalences for a model structure. We do so by leveraging the results from the previous section, which characterized the conditions necessary for a potential pair of weak equivalences and acyclic fibrations to give a model structure.

\begin{example}\label{ex:burgerfailnocenter}  
Any wide decomposable subcategory $Q$ of $[2]\times [1]$ that contains the middle vertical arrow but neither of the other two vertical arrows is neither closed under pushouts nor pullbacks. As explained in \Cref{ex:middle-arrow}, such a set is then not a weak equivalence set. Thus, none of the wide decomposable subcategories shown below are weak equivalence sets.
 
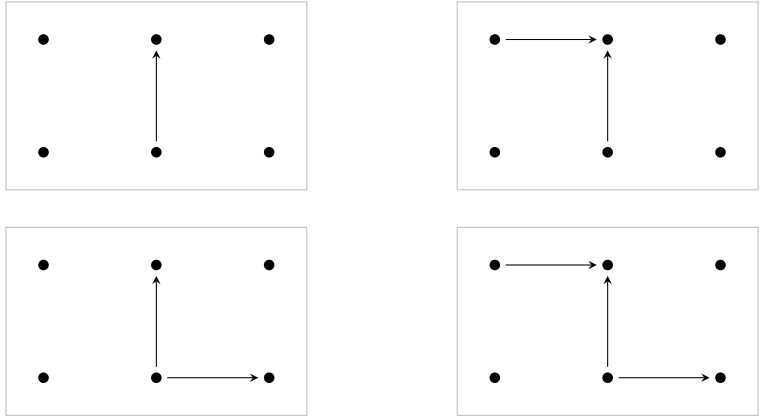
\[ 
\begin{tikzpicture}[>=stealth, bend angle=20, baseline=(current bounding box.center)]
    \fill (-3,0) circle (2pt);
    \fill (-1.5,0) circle (2pt);
    \fill (0,0) circle (2pt);
    \fill (-3,1.5) circle (2pt);
    \fill (-1.5,1.5) circle (2pt);
    \fill (0,1.5) circle (2pt);

    \fill (3,0) circle (2pt);
    \fill (4.5,0) circle (2pt);
    \fill (6,0) circle (2pt);
    \fill (3,1.5) circle (2pt);
    \fill (4.5,1.5) circle (2pt);
    \fill (6,1.5) circle (2pt);

    \fill (-3,-3) circle (2pt);
    \fill (-1.5,-3) circle (2pt);
    \fill (0,-3) circle (2pt);
    \fill (-3,-1.5) circle (2pt);
    \fill (-1.5,-1.5) circle (2pt);
    \fill (0,-1.5) circle (2pt);

    \fill (3,-3) circle (2pt);
    \fill (4.5,-3) circle (2pt);
    \fill (6,-3) circle (2pt);
    \fill (3,-1.5) circle (2pt);
    \fill (4.5,-1.5) circle (2pt);
    \fill (6,-1.5) circle (2pt);
    
    \node (00tl) at (-3,0) {};
    \node (10tl) at (-1.5,0) {};
    \node (20tl) at (0,0) {};
    \node (01tl) at (-3,1.5) {};
    \node (11tl) at (-1.5,1.5) {};
    \node (21tl) at (0,1.5) {};

    \node (00tr) at (3,0) {};
    \node (10tr) at (4.5,0) {};
    \node (20tr) at (6,0) {};
    \node (01tr) at (3,1.5) {};
    \node (11tr) at (4.5,1.5) {};
    \node (21tr) at (6,1.5) {};

    \node (00bl) at (-3,-3) {};
    \node (10bl) at (-1.5,-3) {};
    \node (20bl) at (0,-3) {};
    \node (01bl) at (-3,-1.5) {};
    \node (11bl) at (-1.5,-1.5) {};
    \node (21bl) at (0,-1.5) {};

    \node (00br) at (3,-3) {};
    \node (10br) at (4.5,-3) {};
    \node (20br) at (6,-3) {};
    \node (01br) at (3,-1.5) {};
    \node (11br) at (4.5,-1.5) {};
    \node (21br) at (6,-1.5) {};

    \draw[->] (10tl)--(11tl);

    \draw[->] (10tr)--(11tr);
    \draw[->] (01tr)--(11tr);

    \draw[->] (10bl)--(11bl);
    \draw[->] (10bl)--(20bl);

    \draw[->] (01br)--(11br);
    \draw[->] (10br)--(11br);
    \draw[->] (10br)--(20br);

    \draw[lightgray] (-3.5,-0.5)--(0.5,-0.5)--(0.5,2)--(-3.5,2)--(-3.5,-0.5);
    \draw[lightgray] (2.5,-0.5)--(6.5,-0.5)--(6.5,2)--(2.5,2)--(2.5,-0.5);
    \draw[lightgray] (-3.5,-3.5)--(0.5,-3.5)--(0.5,-1)--(-3.5,-1)--(-3.5,-3.5);
    \draw[lightgray] (2.5,-3.5)--(6.5,-3.5)--(6.5,-1)--(2.5,-1)--(2.5,-3.5);

\end{tikzpicture}
\]
\end{example}

This leads us to an important observation that we  use heavily throughout this section. Recall from \Cref{rem:shortarrowsareeither} that every short arrow which is a weak equivalence has to be either a fibration or cofibration, but not both.

\begin{remark}\label{rem:shortarrows}
As the set $\AF$ is closed under pullbacks and the set $\AC$ is closed under pushouts, for every short arrow $\sigma$ in a weak equivalence set $\W$, either all pushouts of $\sigma$ are in $\W$ or all pullbacks of $\sigma$ are in $\W$.
\end{remark}

 While \Cref{rem:shortarrows} gives a  property of short arrows that must be satisfied in a weak equivalence set, this property is not sufficient to test \textit{if} a given wide decomposable subcategory  is  a weak equivalence set, as the next example shows. 
 
 \begin{example}\label{ex:cond1butnot2}
 Consider the wide decomposable subcategory $Q$ of $[2] \times [2]$ shown below by depicting only its short arrows. These short arrows satisfy the condition that either all their pushouts or all their pullbacks are  in $Q$. However, we will show there is no model structure with weak equivalence class $Q$.
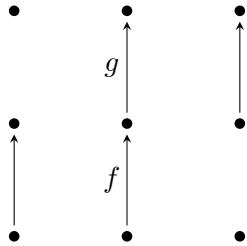
\[
\begin{tikzpicture}[>=stealth, bend angle=20, baseline=(current bounding box.center)]
    \fill (0,0) circle (2pt);
    \fill (1.5,0) circle (2pt);
    \fill (3,0) circle (2pt);
    \fill (0,1.5) circle (2pt);
    \fill (1.5,1.5) circle (2pt);
    \fill (3,1.5) circle (2pt);
    \fill (0,3) circle (2pt);
    \fill (1.5,3) circle (2pt);
    \fill (3,3) circle (2pt);
    
    \node (00) at (0,0) {};
    \node (10) at (1.5,0) {};
    \node (20) at (3,0) {};
    \node (01) at (0,1.5) {};
    \node (11) at (1.5,1.5) {};
    \node (21) at (3,1.5) {};
    \node (02) at (0,3) {};
    \node (12) at (1.5,3) {};
    \node (22) at (3,3) {};

    \draw[->] (00) -- (01);
    \draw[->] (10)--(11);
    \draw[->] (11)--(12);
    \draw[->] (21)--(22);
    \draw[->, white, bend left] (02) to (22);
    \draw[->, white, bend right] (00) to (20);

    \node at (1.3,0.75) {\small $f$};
    \node at (1.3,2.25) {\small $g$};
    \node (00label) at (-0.5,-0.2){\small $(0,0)$};
    \node (21label) at (3.5,3.2){\small $(2,2)$};
\end{tikzpicture}
\]

If $Q$ is a weak equivalence set, then every short arrow in $Q$ is either in $\AC$ or in $\AF$. Since $\AC$ is closed under pushouts and $\AF$ is closed under pullbacks, the arrow $g\colon (1,1) \rightarrow (1,2)$ must be in $\AC$ and the arrow $f\colon (1,0) \rightarrow (1,1)$ must be in $\AF$. 
Consider the middle arrow $(1,0)\rightarrow(1,2)$, noting that $g\circ f$   is its only nontrivial decomposition.
Since $g\circ f\in Q=\AF\circ \AC$, it follows that $g\circ f$ has to be in either $\AC$ or $\AF$. However, neither its pushout nor its pullback is in $Q$, so there can be no model structure with weak equivalences given by $Q$.

In contrast, consider the wide decomposable subcategory $Q$ of $[2] \times [2]$ shown below with the composition arrow omitted. As we will see in the remainder of this section, this is a weak equivalence set. The obstruction seen in the first case is not present here.
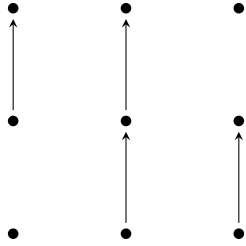
\[
    \begin{tikzpicture}[>=stealth, bend angle=20, baseline=(current bounding box.center)]
    \fill (0,0) circle (2pt);
    \fill (1.5,0) circle (2pt);
    \fill (3,0) circle (2pt);
    \fill (0,1.5) circle (2pt);
    \fill (1.5,1.5) circle (2pt);
    \fill (3,1.5) circle (2pt);
    \fill (0,3) circle (2pt);
    \fill (1.5,3) circle (2pt);
    \fill (3,3) circle (2pt);
    
    \node (00) at (0,0) {};
    \node (10) at (1.5,0) {};
    \node (20) at (3,0) {};
    \node (01) at (0,1.5) {};
    \node (11) at (1.5,1.5) {};
    \node (21) at (3,1.5) {};
    \node (02) at (0,3) {};
    \node (12) at (1.5,3) {};
    \node (22) at (3,3) {};

    \draw[->] (01) -- (02);
    \draw[->] (10)--(11);
    \draw[->] (11)--(12);
    \draw[->] (20)--(21);
    \draw[->, white, bend left] (02) to (22);
    \draw[->, white, bend right] (00) to (20);
\end{tikzpicture}\]

\end{example}

Inspired by these examples, we introduce the following condition.  We  prove that this is sufficient for a wide decomposable subcategory to be a weak equivalence set in \Cref{thm:allaboutw}.

 \begin{conditions}\label{assumption}
 Let $Q$ be a wide decomposable subcategory of a finite lattice $P$. For all morphisms $f$ in $Q$, there exists a factorization $f = \sigma_n \circ \sigma_{n-1} \circ \ldots \circ \sigma_1$ into short arrows such that for some $0 \leq k \leq n$, both of the following hold:
        \begin{itemize}
            \item for any $i \leq k$, all pushouts of $\sigma_i$ are in $Q$, and
            \item for any $i > k$, all pullbacks of $\sigma_i$ are in $Q$.
        \end{itemize}
\end{conditions}

In the particular case when $f$ is a short arrow, \Cref{assumption} implies that either all of its pullbacks or all of its pushouts are in $Q$. This property had already been established for weak equivalence sets in \Cref{rem:shortarrows}.

\begin{example} Recall the two cases of \cref{ex:cond1butnot2}. In the first case, $Q$ satisfies \Cref{assumption} for short arrows, but the composite arrow fails the condition. In the second, $Q$ satisfies \Cref{assumption}. 
\end{example}

\begin{example}
The four wide decomposable subcategories  given in \Cref{ex:burgerfailnocenter} are  the only wide decomposable subcategories of $[2]\times [1]$ that fail \Cref{assumption}.
\end{example}

We now show that if a wide decomposable subcategory $Q$ satisfies \Cref{assumption} then there is a model  structure with weak equivalences $Q$.

\begin{proposition}\label{prop:Tmax-and-W-closed-under-pushouts} Let $Q$ be a wide decomposable subcategory of a finite lattice $P$. If $Q$ satisfies \Cref{assumption}, then ${}^\boxslash \Tmax \cap Q$ is closed under pushouts and hence is a cotransfer system.
\end{proposition}

\begin{proof}
    Let $f\colon x \to z$ be an arrow in $Q$ such that not all pushouts of $f$ are in $Q$. We will show that $f \notin {}^\boxslash \Tmax$, which will imply that for any arrow in ${}^\boxslash \Tmax \cap Q$, all of the pushouts of that arrow are also in ${}^\boxslash \Tmax \cap Q$.
     
\Cref{assumption} implies that we can factor $f=f_2 \circ f_1$, where $f_2$ is a composition of short arrows whose pullbacks are in $Q$ and $f_1$ is a composition of short arrows whose pushouts are in $Q$. By our assumption that not all pushouts of $f$ are in $Q$, we know that $f_2$ is nontrivial. We can write this as the following commutative square.

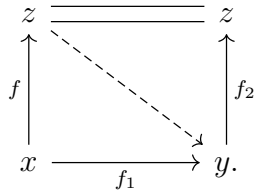
\[\begin{tikzcd}
	z && z \\
	\\
	x && y.
	\arrow[shift left, no head, from=1-1, to=1-3]
	\arrow[shift right, no head, from=1-1, to=1-3]
	\arrow[dashed, from=1-1, to=3-3]
	\arrow["f", from=3-1, to=1-1]
	\arrow["{f_1}"', from=3-1, to=3-3]
	\arrow["{f_2}"', from=3-3, to=1-3]
\end{tikzcd}\]
As $f_2$ is a composite of short arrows whose pullbacks are in $Q$, it follows from \cref{prop:Tmax-closure} that $f_2 \in \Tmax$. Because $y \neq z$, there is no lift in this commutative square, \emph{i.e.}, $f \notin {}^\boxslash \Tmax$, which was our claim. \end{proof} 

Finally, given a wide decomposable subcategory $Q$, not only is \Cref{assumption}  sufficient  for $Q$ to be a weak equivalence set, it is necessary.

\begin{theorem}\label{thm:allaboutw}
    Let $Q$ be a wide decomposable subcategory of a finite lattice $P$. The following are equivalent:
    \begin{enumerate}
    \item there is a model structure with $Q$ as the class of weak equivalences, 
      \item $(Q, \Tmax)$ gives a model structure, and
    \item $Q$ satisfies  \Cref{assumption}.   
    \end{enumerate}
\end{theorem}

\begin{proof}

First,  $(1)$ follows directly from $(2)$. Further,  $(1)$ implies $(2)$ by \Cref{prop:AFmax=Tmax}, and $(3)$ implies $(2)$ by \Cref{prop:Tmax-and-W-closed-under-pushouts} together with  \Cref{prop:cotransfer}.

Lastly, we  show that $(2)$ implies $(3)$. Since $(Q,\Tmax)$ gives a model structure, every short arrow in $Q$ is either a fibration or a cofibration. Since fibrations are closed under pullbacks and cofibrations are closed under pushouts, all short arrows of $Q$ satisfy \Cref{assumption}. We use the factorization axiom to show the result for the general case. Since $(Q,\Tmax)$ is a model structure, by \Cref{cor:dual-AFmin-ACmin}, the corresponding acyclic cofibrations are given by $\AC_{min}$. Thus, if $f$ is a morphism in $Q$ then $f=g\circ h$ where $g\in \Tmax$ and $h\in \AC_{min}$. By \cref{prop:Tmax-saturated}, $\Tmax$ is saturated. Hence, $g$ can be written as a composition of short arrows in $\Tmax$, and by definition, every pullback of every arrow in $\Tmax$ is in $Q$. Similarly, since $h\in \AC_{min} \subseteq \Cmax$, $h$ can be written as a composition of short arrows in $\Cmax$, all of whose pushouts are in $Q$. Thus, we can factor every morphism $f$ in $Q$ following the requirements of \Cref{assumption}.
\end{proof}

\begin{remark}
    When $P=[m]\times [n]$ for $m,n >0$, \Cref{assumption} is easier to visualize, partly because we claim that  \Cref{assumption} only needs to be checked for vertical and horizontal arrows in $Q$.
    
    The nontrivial pullbacks of a short vertical arrow $(i,j)\to(i,j+1)$ are all arrows of the form $(k,j)\to (k,j+1)$ for $k<i$, \emph{i.e.}, the vertical arrows to the left on the same row.  Similarly, the pullbacks of a short horizontal arrow $(i,j)\to(i+1,j)$ are all arrows of the form $(i,k)\to (i+1,k)$ for $k<j$, \emph{i.e.}, the horizontal arrows below in the same column. Therefore, \Cref{assumption} tells us that for every short vertical arrow in $Q$, either all arrows to the left of it are in $Q$ or all arrows to the right of it are in $Q$. Similarly, for every horizontal short arrow in $Q$, either all short arrows above it or below it are also in $Q$. For a long vertical arrow $f$ in $Q$, the condition is equivalent to saying that there exists a horizontal line such that all short vertical arrows in the grid below the line and to the right of $f$  and all short vertical arrows above the line and to the left of $f$ must be in $Q$. For horizontal arrows the condition is similar. 

    Let $f\colon (i,j) \longrightarrow (k,\ell)$ be an arrow in $Q$. We call $\ell-j$ the \emph{vertical height} of $f$. We prove the claim that \Cref{assumption} only needs to be checked for vertical and horizontal arrows by induction on the vertical height.
    
    The base case is the assumption that horizontal arrows satisfy \Cref{assumption}, as these are precisely those arrows with vertical height 0. 
    
    Let $f\colon (i,j) \rightarrow (k,\ell)$ be an arrow in $Q$ of vertical height at least 1. This means that $f$ has a factorization which contains at least one vertical short arrow. More precisely, there is a factorization of $f$ into short arrows such that the last short arrow is vertical. We call this arrow $\sigma$. Since $Q$ is decomposable, all arrows within the rectangle spanned by $f$,
    \[
    \{(x,y) \in P \,\,|\,\, i \leq x \leq k, \, j \leq y \leq \ell \},
    \]
    must also be in $Q$, in particular all short arrows in this rectangle, including $\sigma$, are in $Q$. 
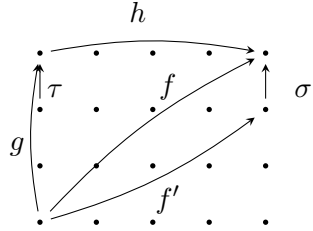
    \[
    \begin{tikzpicture}[>=stealth, bend angle=10, baseline=(current bounding box.center)]
    \fill (0,0) circle (1pt);
    \fill (0.75,0) circle (1pt);
    \fill (1.5,0) circle (1pt);
    \fill (2.25,0) circle (1pt);
    \fill (3,0) circle (1pt);
    \fill (0,0.75) circle (1pt);
    \fill (0.75,0.75) circle (1pt);
    \fill (1.5,0.75) circle (1pt);
    \fill (2.25,0.75) circle (1pt);
    \fill (3,0.75) circle (1pt);
    \fill (0,1.5) circle (1pt);
    \fill (0.75,1.5) circle (1pt);
    \fill (1.5,1.5) circle (1pt);
    \fill (2.25,1.5) circle (1pt);
    \fill (3,1.5) circle (1pt);
    \fill (0,2.25) circle (1pt);
    \fill (0.75,2.25) circle (1pt);
    \fill (1.5,2.25) circle (1pt);
    \fill (2.25,2.25) circle (1pt);
    \fill (3,2.25) circle (1pt);

    \node (00) at (0,0) {};
    \node (10) at (1.5,0) {};
    \node (20) at (3,1.5) {};
    \node (01) at (0,1.5) {};
    \node (11) at (1.5,1.5) {};
    \node (21) at (3,1.5) {};
    \node (02) at (0,2.25) {};
    \node (12) at (1.5,2.25) {};
    \node (22) at (3,2.25) {};
    \node at (3.5,1.75){\small $\sigma$};
    \node at (1.7,1.8){\small $f$};
    \node at (1.7,.3){\small $f'$};
    \node at (1.3,2.8){\small $h$};
    \node at (-.3,1){\small $g$};
     \node at (0.2,1.75){\small $\tau$};

    \draw[->, bend left] (00) to (22);
    \draw[->] (20) to (22);
    \draw[->, bend right] (00) to (20);
    \draw[->, bend left] (00) to (02);
    \draw[->, bend left] (02) to (22);

\draw[->] (01) to (02);

\end{tikzpicture}\]
Either all  pullbacks of $\sigma$ or all  pushouts of $\sigma$ must be in $Q$ by assumption. 

If all the pullbacks are in $Q$, we consider the arrow $f'\colon (i,j)\to (k, \ell-1)$ (see diagram), which by the induction hypothesis has a decomposition that satisfies \Cref{assumption}. Then this decomposition, together with $\sigma$, is a decomposition of $f$ that satisfies \Cref{assumption}.

If not all pullbacks of $\sigma$ are in $Q$, 
consider the  vertical arrow $g\colon (i,j) \rightarrow (i,\ell)$. There is only one way of factoring $g$ into short arrows, and we call the last arrow of this factorization $\tau$.
As not all pullbacks of $\sigma$ are in $Q$, $\tau$ also has a pullback that is not in $Q$.
Therefore, the only way that $g$ can satisfy \Cref{assumption} is by having $Q$ contain all  pushouts of all  short arrows in $g$. But then the decomposition of $g$ followed by the assumed decomposition of the horizontal arrow $h$ gives a decomposition of $f$ satisfying \Cref{assumption}.

This method can also be applied to a $k$-dimensional grid $P$, reducing the requirement for \Cref{assumption} on all arrows in $Q$ to only the vertical and horizontal arrows in $Q$. 

Given any morphism $f$ in $Q$ that is parallel to one of the axes of the grid, \emph{i.e.}, the source and target of $f$ share all coordinates but one, its unique factorization into short arrows satisfies  \Cref{assumption}.
    \end{remark}

\section{Examples and Combinatorics}\label{sec:examples}

In this section, we apply our theoretical results to describe all the weak equivalence sets on $[n]\times [1]$ for arbitrary $n$, and all the model structures on the  iterated parallel composition of $[2]$ with itself and the pentagon $N_5$.

\subsection{Weak equivalences on $[n] \times [1]$}

We  first  count  the weak equivalence sets on $[n] \times [1]$ for any $n \geq 1$. This lattice is of  relevance in equivariant homotopy theory as the subgroup lattice $\mathrm{Sub}(C_{p^nq})$ for $p$ and $q$ distinct primes.

\begin{proposition}
    Let $n\geq 1$.  There are $2^{2n+2}-2^{n+1}-2^nn$ weak equivalence sets on the lattice $[n]\times [1]$.
\end{proposition}

\begin{proof} We first describe the form  weak equivalence sets on $[n] \times [1]$ must take. Since weak equivalence sets are decomposable, we can fully describe a weak equivalence set $\W$ in terms of its short arrows, since by \Cref{thm:allaboutw}, \Cref{assumption} is satisfied. Hence, the vertical arrows of $\W$ must follow one of the following criteria:
    \begin{itemize}
    \item none of the the vertical arrows are present,
    \item all vertical arrows are present, or    
    \item if some are present and others are not then either:
    \begin{itemize}
    \item the ones present are all stacked to the left,
    \item the ones present are all stacked to the right, or
    \item there is exactly one gap of any size in the middle.
    \end{itemize}
    \end{itemize}

\Cref{assumption} does not impose any condition on the horizontal arrows of $\W$ because if a horizontal arrow of $\W$ is in the top row of $[n]\times [1]$ then there are no nontrivial pushouts of the arrow and if it is in the bottom row then there are no nontrivial pullbacks. Instead, requirements on the horizontal arrows stem from the layout of the vertical arrows, specifically whether or not two neighboring vertical arrows are in $\W$.

Since $\W$ is closed under composition, if $\W$ contains two neighboring vertical arrows and one of the horizontal arrows between them, it contains the diagonal arrow between them. The decomposability of $\W$ then forces the other horizontal arrow to be in $\W$. Thus, the two horizontal arrows between neighboring vertical arrows are either both in $\W$ or both not in $\W$.

Now suppose that a vertical arrow is in $\W$ but one of its neighboring vertical arrows is not. For example, suppose $(i,0)\to (i,1)$ is in $\W$ but $(i+1,0) \to (i+1,1)$ is not. Then the arrow $(i,1)\to (i+1,1)$ cannot be in $\W$, otherwise, composition and decomposability would imply $(i+1,0)\to (i+1,1)$ is in $\W$.  However, there are no restrictions on the other horizontal arrow $(i,0) \to (i+1,0)$, see the diagram below for an illustration.

    
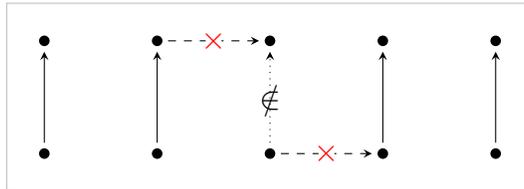
\[
     \begin{tikzpicture}[>=stealth, bend angle=20, baseline=(current bounding box.center), decoration=crosses]
    \node (00) at (0,0) {};
    \node (10) at (1.5,0) {};
    \node (20) at (3,0) {};
    \node (30) at (4.5,0){};
    \node (40) at (6,0){};
    \node (01) at (0,1.5) {};
    \node (11) at (1.5,1.5) {};
    \node (21) at (3,1.5) {};
    \node (31) at (4.5,1.5){};
    \node (41) at (6,1.5){};
    
    \fill (0,0) circle (2pt);
    \fill (1.5,0) circle (2pt);
    \fill (3,0) circle (2pt);
    \fill (4.5,0) circle (2pt);
    \fill (6,0) circle (2pt);
    \fill (0,1.5) circle (2pt);
    \fill (1.5,1.5) circle (2pt);
    \fill (3,1.5) circle (2pt);
    \fill (6,1.5) circle (2pt);
    \fill (4.5,1.5) circle (2pt);

    \draw[->,dashed] (11)--(21);
    \draw[->,dashed] (20)--(30); 
    \fill[white] (2.25,1.5) circle (3pt);
    \fill[white] (3.75,0) circle (3pt);
    
    \node[red] at (2.25,1.5){$\times$};
    \node[red] at (3.75,0){$\times$};
    \node at (3,0.75){$\notin$};

    \draw[->] (00)--(01);
    \draw[->] (10)--(11);
    \draw[->] (30)--(31);
    \draw[->] (40)--(41);
    \draw[dotted,->] (20) to (21);

    \draw[lightgray] (-0.5,-0.5)--(6.5,-0.5)--(6.5,2)--(-0.5,2)--(-0.5,-0.5);
\end{tikzpicture}
    \]
The case when $(i,0)\to (i,1)$ is in $\W$ but $(i-1,0)\to (i-1,1)$ is not is similar.
There are no restrictions on horizontal arrows between neighboring vertical arrows that are not in $\W$. It is possible for none, one, or both to be in $\W$. Below is a picture of a weak equivalence set in  $[8] \times[1]$ with the composition arrow omitted which shows all the possible combinations of horizontal and vertical arrows. 
 
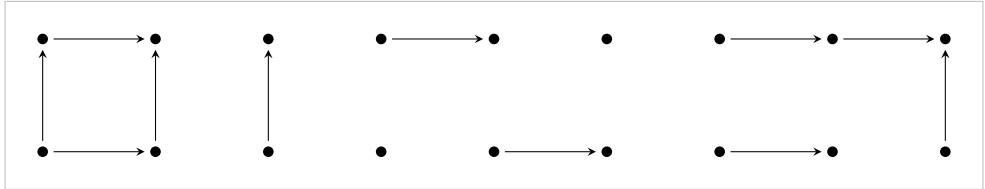
\[
     \begin{tikzpicture}[>=stealth, bend angle=20, baseline=(current bounding box.center)]
    \fill (0,0) circle (2pt);
    \fill (1.5,0) circle (2pt);
    \fill (3,0) circle (2pt);
    \fill (4.5,0) circle (2pt);
    \fill (6,0) circle (2pt);
    \fill (7.5,0) circle (2pt);
    \fill (9,0) circle (2pt);
    \fill (10.5,0) circle (2pt);
    \fill (12,0) circle (2pt);
    \fill (0,1.5) circle (2pt);
    \fill (1.5,1.5) circle (2pt);
    \fill (3,1.5) circle (2pt);
    \fill (4.5,1.5) circle (2pt);
    \fill (6,1.5) circle (2pt);
    \fill (7.5,1.5) circle (2pt);
    \fill (9,1.5) circle (2pt);
    \fill (10.5,1.5) circle (2pt);
    \fill (12,1.5) circle (2pt);

    \node (00) at (0,0){};
    \node (10) at (1.5,0){};
    \node (20) at (3,0){};
    \node (30) at (4.5,0){};
    \node (40) at (6,0){};
    \node (50) at (7.5,0){};
    \node (60) at (9,0){};
    \node (70) at (10.5,0){};
    \node (80) at (12,0){};
    \node (01) at (0,1.5){};
    \node (11) at (1.5,1.5){};
    \node (21) at (3,1.5){};
    \node (31) at (4.5,1.5){};
    \node (41) at (6,1.5){};
    \node (51) at (7.5,1.5){};
    \node (61) at (9,1.5){};
    \node (71) at (10.5,1.5){};
    \node (81) at (12,1.5){};
    
    \draw[->] (00)--(01);
    \draw[->] (10)--(11);
    \draw[->] (20)--(21);
    \draw[->] (00)--(10);
    \draw[->] (01)--(11);
    \draw[->] (31)--(41);
    \draw[->] (40)--(50);
    \draw[->] (60)--(70);
    \draw[->] (61)--(71);
    \draw[->] (71)--(81);
    \draw[->] (80)--(81);
    \draw[->, white, bend right] (00) to (20);
    \draw[->, white, bend left] (01) to (21);

    \draw[lightgray] (-0.5,-0.5)--(12.5,-0.5)--(12.5,2)--(-0.5,2)--(-0.5,-0.5);
 \end{tikzpicture}
 \]

    We can now proceed to count weak equivalence sets based on the three cases for the vertical arrows imposed by \Cref{assumption}. In the case where there are no vertical arrows, any combination of horizontal arrows creates a weak equivalence set. This gives $2^{2n}$ weak equivalence sets. In the case where all vertical arrows are there, there are $2^n$ weak equivalence sets.
    
    Now suppose the vertical arrows present are all stacked to the left. More precisely, there exists $0\leq \ell < n$ such that for all $i\leq \ell$, $(i,0)\to (i,1)$ is in $\W$, and for $i>\ell$, $(i,0)\to (i,1)$ is in not $\W$. The considerations above imply that in this case there are $2^{2n-\ell-1}$ combinations of horizontal arrows that lead to weak equivalence sets. Taking the sum over $\ell$, we get a total of
    \[\sum_{\ell=0}^{n-1} 2^{2n-\ell-1}=2^{2n}-2^n\] weak equivalences sets.
    The case when the vertical arrows are stacked to the right is analogous and gives the same count. 
    
    Finally, we consider the case in which there is a gap in the middle of the vertical arrows. Thus, suppose there exist $\ell, k\geq 0$ with $n-k\geq \ell +2$ such that a vertical arrow $(i,0)\to (i,1)$ is in $\W$ if and only if $i\leq \ell$ or $i \geq n-k$. This means that $\W$ contains the $\ell+1$ leftmost vertical arrows and the $k+1$ rightmost vertical arrows, with a gap in between. The considerations above imply that there are $2^{2n-\ell-k-2}$ configurations of horizontal arrows that give a valid $\W$. Taking the sum over $\ell$ and $k$ gives a total of
    \[\sum_{k=0}^{n-2}\sum_{\ell=0}^{n-k-2}2^{2n-\ell-k-2}=2^{2n}-2^{n+1}-2^n(n-1). \]

Adding up all  cases, it follows that the total count of weak equivalence sets on $[n]\times [1]$ is $2^{2n+2}-2^{n+1}-2^nn$, as claimed. \end{proof}

\begin{remark}
    At the time of writing, counting transfer systems on $[n]\times [1]$ has been untractable, thus counting model structures on $[n]\times[1]$ is untractable as well. We expect this to be a topic of future research. 
\end{remark}

\subsection{Model structures on $[2]^{\ast n}$}\label{subsec:diamond}

We  now  provide a complete characterization of the model structures on $[2]^{\ast n}$. This lattice is of  relevance in equivariant homotopy theory as $[2]^{\ast(p+1)}$  is the subgroup lattice $\mathrm{Sub}(C_p \times C_p)$ of a rank two elementary abelian group  \cite[Lemma 5.1]{Baoetal}. 

Recall from \Cref{ex:2iterated} that $[2]^{\ast n}$ consists of a maximal element $\top$, a minimal element $\bot$ and $n$ incomparable elements labeled $1$ through $n$. We first use \Cref{assumption} to determine all weak equivalence sets on $[2]^{\ast n}$. Then given a weak equivalence set $\W$ we use \Cref{thm:AFW-is-lattice} to count all possible model structures with weak equivalences $\W$.

Before proceeding, we need to know the pushouts and pullbacks of $i \to \top$ and $\bot \to i$ for any $i$. First, $i \to \top$ has no nontrivial pushouts, and if $j \neq i$ then $\bot \to j$ is a pullback of $i \to \top$. Similarly, $\bot \to i$ has no nontrivial pullbacks, and if $i\neq j$ then $j \to \top$ is a pushout of $\bot \to j$. For example, on the left in the figure below the dashed arrows are  the pullbacks of the solid arrow $2\rightarrow\top$. On the right, the dashed arrows are  the pushouts of the solid arrow $\bot \rightarrow 3$.
     
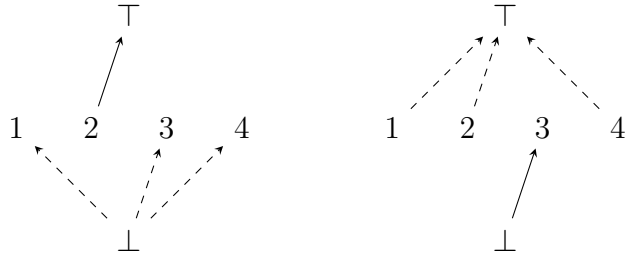
\[
\begin{tikzpicture}[>=stealth, bend angle=20, baseline=(current bounding box.center)]
    \node (botl) at (0,0){$\bot$};
    \node (topl) at (0,3){$\top$};
    \node (1l) at (-1.5,1.5){$1$};
    \node (2l) at (-0.5,1.5){$2$};
    \node (3l) at (0.5,1.5){$3$};
    \node (4l) at (1.5,1.5){$4$};

    \node (botr) at (5,0){$\bot$};
    \node (topr) at (5,3){$\top$};
    \node (1r) at (3.5,1.5){$1$};
    \node (2r) at (4.5,1.5){$2$};
    \node (3r) at (5.5,1.5){$3$};
    \node (4r) at (6.5,1.5){$4$};

    \draw[->, dashed] (botl)--(1l);
    \draw[->, dashed] (botl)--(3l);
    \draw[->, dashed] (botl)--(4l);
    \draw[->] (2l)--(topl);

    \draw[->] (botr)--(3r);
    \draw[->, dashed] (1r)--(topr);
    \draw[->, dashed] (2r)--(topr);
    \draw[->, dashed] (4r)--(topr);
\end{tikzpicture}
\]

\begin{proposition}
    Let $n \geq 1$. There are $3^n+1$ weak equivalence sets on the lattice $[2]^{* n}$.
\end{proposition}
\begin{proof}

 First we show that every wide decomposable subcategory $\W$ of $[2]^{* n}$ is a weak equivalence set. Given $i=1,\dots,n$, if $\bot \to i$ and $i\to \top$ are both in $\W$, then the composite $\bot \to \top$ is as well, and by decomposability, $\W$ is complete. Therefore, if $\W$ is not complete, it can contain at most one  of $\bot \to i$ and $i\to \top$ for each $i$. For any $i$, $\bot \to i$ and $i\to \top$ have no non-trivial pullbacks or pushouts, respectively, therefore \Cref{assumption} is satisfied, and $\W$ is a weak equivalence set.

 When counting the non-complete weak equivalence sets, we  have three independent choices for each $i$: $\W$ contains $i \to \top$, $\bot \to i$, or it contains neither. Thus, including the complete subcategory, there are $3^n+1$  options for $\W$.
 \end{proof}

   We can now count the model structures on $[2]^{* n}$. 
      \begin{proposition}\label{prop:we on sandwich}
    Let $n \geq 1$. There are $3^n+2^{n+1}+3n$ model structures on $[2]^{* n}$.
\end{proposition}
    \begin{proof}
    When $n=1$ the results of \cite{ModelStructuresOnFiniteTotalOrders}   give 10 model structures, which satisfies the formula. When $n=2$, there are 10 weak equivalence sets on $[2]^{* 2}$. The complete weak equivalence set gives rise to 10 model structures, since by \Cref{rem:Wcotransfer} each transfer system in $[2]^{*2}$ gives rise to a unique model structure. The two weak equivalence sets that are cotransfer systems and not complete each contain three transfer systems. Lastly, there is exactly one model structure on  each remaining weak equivalence set. This gives a total of 23 model structures, which again satisfies the formula.

Now let $n \ge 3$, and let $\W$ be a weak equivalence set. If $\W$ is complete, then there is model structure for each transfer system. As computed in \cite[Proposition 5.2]{Baoetal}, the number of transfer systems in $[2]^{*n}$ is $2^{n+1}+n$. 
    
    Now assume that $\W$ is not complete. By \Cref{thm:AFW-is-lattice}, we need to determine $\AF_{max}$, $\AF_{min}$, and all transfer systems between them. We consider three cases.

    \emph{Case 1.} Suppose  $\W$ consists of $i \to \top$ for some $1 \leq i \leq n$, and $\bot \to j$  and all $j$ such that $i \neq j$.
    Note that $\W$ is closed under pullbacks, so it follows that $\AF_{max}=\W$. To determine $\AF_{min}$, we start with $\AC_{max}$, which cannot contain any arrows of the form $\bot \to j$, because the assumption that $n\geq 3$ implies that there is a pushout of $\bot \to j$ of the form $k\to\top$ that is not in $\W$.  Thus $\AC_{max}$ consists only of $i \to \top$. A direct calculation using \Cref{cor:dual-AFmin-ACmin} shows that $\AF_{min}$ consists of the arrows of the form $\bot \to j$ that are in $\W$. 
    In particular, $\AF_{min}$ and $\AF_{max}$ differ  by exactly one arrow, so there are exactly two model structures with weak equivalences $\W$.

    \emph{Case 2.}  Suppose  $\W$ consists of $\bot \to i$ for some $1 \leq i \leq n$, and $j \to \top$ for all $j$ such that $i \neq j$.
    This case is dual to Case 1, and an analogous argument shows there are exactly two model structures with weak equivalences $\W$.

    \emph{Case 3.} Lastly, we  consider $\W$  that  is not complete and that does not look like Case 1 or Case 2. This implies that either there exists an  $i$ such that both $\bot \to i$ and $i \to \top$ are \emph{not} in $\W$ as on the left below, or for each $i$ either $\bot \to i$ or $i \to \top$ is in $\W$ and at least two arrows go to $\top$ and at least two arrows come from $\bot$ as on the right below.  
    
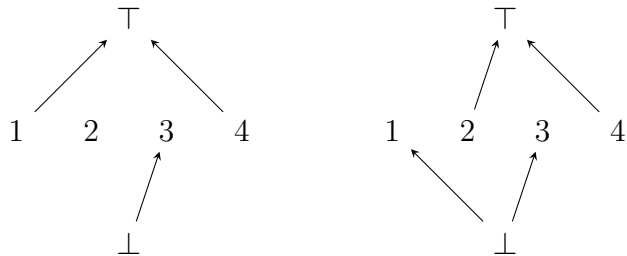
\[
\begin{tikzpicture}[>=stealth, bend angle=20, baseline=(current bounding box.center)]
    \node (botl) at (0,0){$\bot$};
    \node (topl) at (0,3){$\top$};
    \node (1l) at (-1.5,1.5){$1$};
    \node (2l) at (-0.5,1.5){$2$};
    \node (3l) at (0.5,1.5){$3$};
    \node (4l) at (1.5,1.5){$4$};

    \node (botr) at (5,0){$\bot$};
    \node (topr) at (5,3){$\top$};
    \node (1r) at (3.5,1.5){$1$};
    \node (2r) at (4.5,1.5){$2$};
    \node (3r) at (5.5,1.5){$3$};
    \node (4r) at (6.5,1.5){$4$};

    \draw[->] (botl)--(3l);
    \draw[->] (1l)--(topl);
    \draw[->] (4l)--(topl);

    \draw[->] (botr)--(1r);
    \draw[->] (botr)--(3r);
    \draw[->] (2r)--(topr);
    \draw[->] (4r)--(topr);
\end{tikzpicture}
\]
    
    In either case, $\AF_{max}$ cannot contain any of the arrows of the form $j\to \top$, otherwise it would not be closed under pullbacks. Thus $\AF_{max}$ consists precisely of all the arrows in $\W$ of the form $\bot \to j$. A dual argument shows that $\AC_{max}$ consists of the arrows in $\W$ of the form $j\to \top$. A direct calculation  shows that $\AF_{min}=\AF_{max}$, giving  exactly one model structure with weak equivalences $\W$.

Finally, we count the number of weak equivalence sets in each case. There are exactly $n$ weak equivalence sets in each of Case 1 and Case 2, given by the choice of $i$. \cref{prop:we on sandwich} gives a total of $3^n$ non-complete weak equivalence sets, and hence there are $3^n-2n$ in Case 3. Thus, the total number of model structures on $[2]^{*n}$ is
\[(2^{n+1}+n)+2(2n)+(3^n-2n)=3^n+2^{n+1}+3n.\]\end{proof}

\subsection{Model structures on the non-modular lattice, the pentagon $N_5$}

In this final section we explain how to determine all model structures on the pentagon $N_5$, shown below.

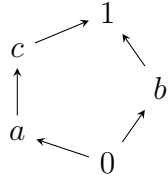
\[
\begin{tikzpicture}[>=stealth, bend angle=20, baseline=(current bounding box.center)]
    \node (0) at (0,0) {$0$};
    \node (a) at (-1.2,0.4) {$a$};
    \node (b) at (0.7,1) {$b$};
    \node (c) at (-1.2,1.5) {$c$};
    \node (1) at (0,2) {1};

    \draw[->] (0)--(a);
    \draw[->] (a)--(c);
    \draw[->] (c)--(1);
    \draw[->] (0)--(b);
    \draw[->] (b)--(1);
\end{tikzpicture}
\]

In the context of transfer systems and model structures, $N_5$ is an interesting example because it is not modular\footnote{
A lattice $P$ is \emph{modular} if 
for every $x,y,z \in P$, if $x \leq z$ then
$x \vee (y \wedge z) = (x \vee y) \wedge z$.
Another characterization is that a lattice is modular if and only if it does not contain $N_5$ as a sublattice.}. In non-modular lattices, and in particular in $N_5$, there are short arrows with pullbacks that are not short, and similar for pushouts.

\begin{proposition} 
On the pentagon $N_5$, every wide decomposable subcategory is a weak equivalence set, and there are 22 weak equivalence sets. Furthermore, there are 70 model structures. 
\end{proposition}

\begin{proof}
In order to prove our claim, we first determine the possible weak equivalence sets. Then, we apply the methods given at the end of Section \ref{sec:FromTStoMS} to calculate $\AF(\W)$ for each weak equivalence set $\W$. This is a straightforward yet laborious task, therefore, instead of presenting it in detail, we give the most important steps as well as some sample calculations.   

We start by showing that a wide decomposable subcategory $\W$ is a weak equivalence set by showing that $\W$ satisfies \Cref{assumption}. For this, we look at the table of the pushouts and pullbacks of all morphisms in $N_5$, shown below, where \ding{63} denotes that there are no nontrivial pushouts or pullbacks. 
The morphisms in the first block of the table are the short arrows, and the morphisms in the second block have a nontrivial factorization.
 We do not include the map $0 \rightarrow 1$ in the table as any decomposable subcategory containing this map is the entire lattice. The pullbacks of $0 \rightarrow 1$ are all arrows of the form $0 \rightarrow x$, and similarly all pushouts of $0 \rightarrow 1$ are all arrows of the form $x \rightarrow 1$.

\begin{center}
\begin{tabular}{|l |l |l| }
\hline
 & pushouts & pullbacks \\
 \hline
 $0 \rightarrow a$ & $b \rightarrow 1$ & \ding{63}\\
 $a \rightarrow c$ & \ding{63} & \ding{63} \\
 $c \rightarrow 1$ & \ding{63} & $0 \rightarrow b$ \\
 $0 \rightarrow b$ & $a \rightarrow 1, c \rightarrow 1$& \ding{63} \\
 $b \rightarrow 1$ & \ding{63} & $0 \rightarrow a, 0 \rightarrow c$ \\
 \hline
 $a \rightarrow 1 $ & $c \rightarrow 1$ & $0 \rightarrow b, a \rightarrow c$ \\
 $0 \rightarrow c$ & $a \rightarrow c, b \rightarrow 1$ & $0 \rightarrow a$ \\
 \hline
\end{tabular}
\end{center}

 Given any wide decomposable subcategory $\W$, we observe the following.
 \begin{itemize}
 \item 
 No short arrow has both nontrivial pushouts and nontrivial pullbacks. Therefore, every short arrow has either all its pullbacks or all its pushouts in $\W$.
 \item Since $\W$ is decomposable, if $a \rightarrow 1$ is in $\W$ then both $a \rightarrow c$ and $c \rightarrow 1$ are in $ \W$. Therefore, $a \rightarrow 1$ satisfies \Cref{assumption}.
 \item Similarly, if $0 \rightarrow c$ is in $ \W$ then both $0 \rightarrow a$ and $a \rightarrow c$ are in $ \W$. Thus, $0 \rightarrow c$ satisfies \Cref{assumption}.
 \end{itemize}

It follows that every wide decomposable subcategory is a weak equivalence set. Direct observation then shows that there are 22 such  subcategories.  Since there are so few morphisms in $N_5$ and transfer systems are closed under pullbacks, we can use the above table of pullbacks to  explicitly 
write down the 26 transfer systems on $N_5$. By duality, there are therefore also 26 cotransfer systems on $N_5$.

We can now directly determine the model structures on each weak equivalence set using the tools of \Cref{sec:FromTStoMS}. In particular, eight of the weak equivalence sets are transfer systems and not cotransfer systems, eight are cotransfer systems and not transfer systems,  five are both transfer systems and cotransfer systems, and one is neither a transfer system nor a cotransfer system. For the 21 weak equivalence sets that are transfer/cotransfer systems we can use \Cref{prop:cotransferW}  and \Cref{rem:Wcotransfer} to determine all model structures.  
Alternatively, recall that $AF_{max}$ is the maximum transfer system in the weak equivalence set. We can determine $\AF_{min}$   and therefore the interval between $\AF_{min}$ and $\AF_{max}$ using \Cref{cor:dual-AFmin-ACmin} and the subsequent steps:
\begin{enumerate}
\item find $\AC_{max}$,
\item determine the corresponding transfer system $\AC_{max}^\boxslash$, and
\item calculate $\AF_{min} = \AC_{max}^\boxslash \cap \W$.
\end{enumerate}

If we follow this procedure for all weak equivalence sets, we see that there are  70  model structures on the pentagon $N_5$. \end{proof}

The following two examples illustrate this process.

\begin{example}
Take $\W$ to be the weak equivalence set shown below. 
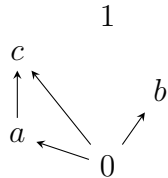
\[
\begin{tikzpicture}[>=stealth, bend angle=20, baseline=(current bounding box.center)]
    \node (0) at (0,0) {$0$};
    \node (a) at (-1.2,0.4) {$a$};
    \node (b) at (0.7,1) {$b$};
    \node (c) at (-1.2,1.5) {$c$};
    \node (1) at (0,2) {1};

    \draw[->] (0)--(a);
    \draw[->] (a)--(c);
    \draw[->] (0)--(c);
    \draw[->] (0)--(b);
\end{tikzpicture}
\]
The largest cotransfer system inside this $\W$ is
$\AC_{max}= \{ a \rightarrow c \}$.
Using \Cref{prop:liftformula}, we can calculate
$$\AC_{max}^\boxslash=\{ 0 \rightarrow a, 0 \rightarrow b, 0 \rightarrow c, 0 \rightarrow 1, b \rightarrow 1, c \rightarrow 1 \},$$ and thus 
$\AF_{min} = \AC_{max}^\boxslash \cap W = \{ 0 \rightarrow a, 0 \rightarrow b, 0 \rightarrow c \}$.

Since $\W$ is a transfer system, $\AF_{max} = \W$. Moreover, there is no transfer system between $\AF_{min}$ and $\AF_{max}$. Therefore,
\[\AF(\W)= \{ \AF_{min}, \W \},\]
\emph{i.e.}, there are  two model structures with this weak equivalence set. 

\end{example} 

\begin{example}
The only weak equivalence set $\W$ which is neither a transfer system nor a cotransfer system is shown below.
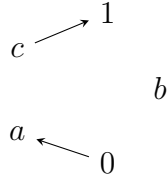
\[
\begin{tikzpicture}[>=stealth, bend angle=20, baseline=(current bounding box.center)]
    \node (0) at (0,0) {$0$};
    \node (a) at (-1.2,0.4) {$a$};
    \node (b) at (0.7,1) {$b$};
    \node (c) at (-1.2,1.5) {$c$};
    \node (1) at (0,2) {1};

    \draw[->] (0)--(a);
    \draw[->] (c)--(1);
\end{tikzpicture}
\]

The maximum transfer system contained in $\W$ is $\Tmax=\AF_{max}=\{0\to a\}$. The maximum cotransfer system is 
 $\Cmax=\AC_{max} = \{ c \rightarrow 1 \}$. Then 
 \[\AC_{max}^\boxslash = \{0 \rightarrow a, 0 \rightarrow b, 0 \rightarrow c, 0 \rightarrow 1, a \rightarrow c, b \rightarrow 1\},\]
 and hence $\AF_{min} =  \{ 0 \rightarrow a \}=\AF_{max} $. Therefore, there is only one model structure with this weak equivalence set.

\end{example}

\bibliographystyle{alpha}
\bibliography{Sub_bibby.bib}

\end{document}